\documentclass{article}
\usepackage{amsthm,amsmath,amssymb,tikz-cd,mathtools,comment}
\usepackage[margin=2cm]{geometry}
\usepackage{hyperref}
\usepackage{bbm}
\usepackage[version=4]{mhchem}
\usepackage{cite}

\newtheoremstyle{Normal}{}{}{}{}{\bfseries}{:}{.5em}{}
\theoremstyle{Normal}

\newtheorem{theorem}{Theorem}
\newtheorem*{theorem*}{Theorem}
\newtheorem{remark}[theorem]{Remark}
\newtheorem{prop}[theorem]{Proposition}
\newtheorem{lemma}[theorem]{Lemma}
\newtheorem{example}[theorem]{Example}
\newtheorem{cor}[theorem]{Corollary}
\newtheorem{conj}[theorem]{Conjecture}
\newtheorem{definition}[theorem]{Definition}
\newtheorem*{setting*}{Switching, Asymptotically Linear}

\newcommand{\PP}{\mathbb P}
\newcommand{\RR}{\mathbb R}
\newcommand{\ZZ}{\mathbb Z}
\newcommand{\E}{\mathbb E}

\newcommand{\supp}{\operatorname{supp}}

\renewcommand{\emptyset}{\varnothing}

\DeclarePairedDelimiter\abs{\lvert}{\rvert}%
\DeclarePairedDelimiter\norm{\lVert}{\rVert}%
\DeclarePairedDelimiter\ceil{\lceil}{\rceil}%
\DeclarePairedDelimiter\floor{\lfloor}{\rfloor}%
\makeatletter
\let\oldabs\abs
\def\abs{\@ifstar{\oldabs}{\oldabs*}}
\let\oldnorm\norm
\def\norm{\@ifstar{\oldnorm}{\oldnorm*}}
\let\oldceil\ceil
\def\ceil{\@ifstar{\oldceil}{\oldceil*}}
\let\oldfloor\floor
\def\floor{\@ifstar{\oldfloor}{\oldfloor*}}
\makeatother

\title{Stability of randomly switching stochastic reaction networks with asymptotically linear transition rates}
\author{Daniele Cappelletti\thanks{\href{mailto:daniele.cappelletti@polito.it}{daniele.cappelletti@polito.it},
support from MUR PRIN grant number 2022XRWY7W.} , Aidan S. Howells\thanks{\href{mailto:aidan.howells@polito.it}{aidan.howells@polito.it},
support from MUR PRIN grant number 2022XRWY7W. Corresponding author.}, and Chuang Xu\thanks{\href{chuangxu@hawaii.edu}{chuangxu@hawaii.edu}, support from Simons Foundation and a start-up funding from the University of Hawai'i at M\={a}noa.} }

\begin{document}

\maketitle

\begin{abstract}
Stochastic reaction networks are mathematical models frequently used in, but not limited to, biochemistry. These models are continuous-time Markov chains whose transition rates depend on certain parameters called \emph{rate constants}, which despite the name may not be constant in real-world applications. In this paper we study how random switching between different stochastic reaction networks with asymptotically linear rate functions affects the stability of the process. We give matrix conditions for both positive recurrence (indeed, exponentially ergodicity) and transience (indeed, evanescence) in both the regime with high switching rates and the regime with low switching rates. We then make use of these conditions to provide examples of processes whose stability behavior changes as the switching rate varies. We also explore what happens in the middle regime where the switching rates are neither high nor low and our theorems do not apply. Specifically, we show by examples that there can be arbitrarily many phase transitions between exponentially ergodicity and evanescence as the switching rate increases.
\end{abstract}

\section{Introduction}\label{sec:Intro}

A Chemical Reaction Network (CRN) is a set of possible chemical species, along with a finite set of chemical transformations that the species can undergo. These models are frequently used in biochemistry both for investigating existing systems and for designing new synthetic biological circuits. Their use also extends to epidemiology, ecology, genetics. It is impossible to cite all the applications present in the scientific literature; we only suggest a few works in the aforementioned directions and some comprehensive books: \cite{gurbuz2022analysis, Gupta_Khammash_2019, khudabukhsh2020survival, swain1984handbook, erdi1989mathematical, anderson2011continuous, ewens2004mathematical, bruno2024analysis}.
We also note that dynamical models of CRNs can be used as an alternative expression for families of polynomial differential equations or for stochastic models such as instances of continuous-time branching processes and Markovian queuing systems.

Deterministic and stochastic mass-action models of CRNs have been studied for over 50 years \cite{Horn_Jackson_1972, Kurtz_1972}. These models are not fully determined by the set of possible reactions, but require an additional ``rate constant" for each reaction, whose chemical interpretation is the propensity of a collection of the reactants to actually undergo the reaction. These rate constants are not a property solely of the reactions themselves, as they are also affected by factors such as temperature \cite[Section~1.6]{Cornish-Bowden_2014}, and the availability of input species which are not modeled and assumed to be in abundance \cite{Goldbeter_Koshland_1981}.

In regimes where the temperature or other relevant factors vary, it may be tempting to assume that the qualitative behavior of the system is well approximated by ignoring the variability of the rate constants, either by taking an average or by separately studying the system in the different regimes. However, recent work suggests that this assumption is dangerous.
For example, consider piecewise-deterministic Markov processes (PDMPs) that take the form of models switching between two linear 2-dimensional ODEs. In \cite{Benaim_LeBorgne_Malrieu_Zitt_2014} they demonstrated that such PDMPs could potentially be stable or unstable, depending only on the transition rates between the two ODEs.
Notably, Example 1.3 in \cite{Benaim_LeBorgne_Malrieu_Zitt_2014} consists of a PDMP that switches between two linear ODEs, each of which can be realized as a deterministic mass-action CRN. Thus, though this was not the point of the paper, it demonstrated that deterministic mass-action CRNs whose rate constants vary stochastically can be either stable or unstable depending only on the rate at which the rate constants vary, and thus the variability of the rate constants cannot be neglected. Similar conclusions could be drawn for deterministically switching models from the previous paper \cite{Balde_Boscain_Mason_2009}, that \cite{Benaim_LeBorgne_Malrieu_Zitt_2014} builds upon. A survey of results on models with deterministic switching between linear ODEs can be found in \cite{Lin_Panos_2009}. PDMPs which switch between nonlinear ODEs have also been studied \cite{Benaim_Strickler_2019}, as have diffusions in random environments \cite{Pinsky_Scheutzow_1992}.

\begin{figure}
\begin{equation*}
\mathcal{R}_1\colon
\begin{tikzcd}
    S_1\arrow[yshift=3]{r}{0.99}&0\arrow[yshift=-3]{l}{1}\\
    S_1\arrow{r}{0.01}&4S_2\\
    S_2\arrow[yshift=3]{r}{0.01}&0\arrow[yshift=-3]{l}{1}\\
    S_2\arrow{r}{0.99}&4S_1
\end{tikzcd}
\qquad\qquad
\mathcal{R}_2\colon
\begin{tikzcd}
    S_1\arrow[yshift=3]{r}{0.01}&0\arrow[yshift=-3]{l}{1}\\
    S_1\arrow{r}{0.99}&4S_2\\
    S_2\arrow[yshift=3]{r}{0.99}&0\arrow[yshift=-3]{l}{1}\\
    S_2\arrow{r}{0.01}&4S_1
\end{tikzcd}
\end{equation*}
\caption{On the left, a chemical reaction network $\mathcal R_1$ with species $S_1$ and $S_2$. Each species can (destructively) produce four of the other, and each species also has associated inflow and degradation reactions. We assume mass-action kinetics. The rate constants of each reaction are written over the reaction arrows. On the right, a second network $\mathcal R_2$, which in this case is just $\mathcal R_1$ but with the roles of the two species reversed. A simple example of a CRN in a stochastic environment is one which evolves either according to $\mathcal R_1$ or $\mathcal R_2$, switching between them at times which are independent of everything else and exponential with parameter $\kappa$ for some fixed $\kappa>0$. We will revisit this model in Example \ref{ex:basic switching transient}.}
\label{fig:intro example}
\end{figure}
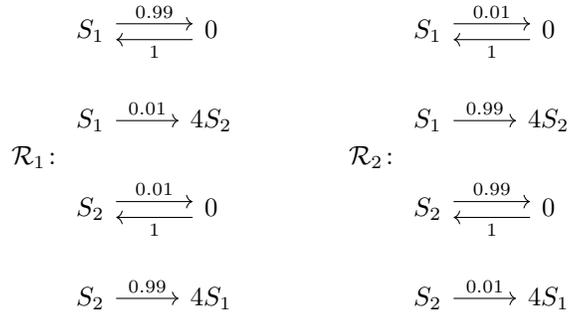

In the present paper, we investigate the stability of a stochastic process that randomly switches between a finite number of stochastic CRNs. We denote a choice of stochastic CRN as an ``environment,'' and we assume the choice changes in time according to a finite continuous-time Markov chain. We further assume that the transition rates of the environment Markov chain are multiplied by a common factor $\kappa>0$ (the bigger $\kappa$ is, the faster the environment switches occur). See Figure \ref{fig:intro example} for an example of this. Our four main theorems give matrix conditions for stability and instability for $\kappa$ large (Theorems \ref{thm:fast-switching-ergodicity} and \ref{thm:fast-switching-transience}, respectively) and similarly for $\kappa$ small (Theorems \ref{thm:slow-switching-ergodicity} and \ref{thm:slow-switching-transience}, respectively). In a sense, our results address the question above of when is it safe to approximate the stability of the system by ignoring the variability of the rate constants. We prove that when the rate constants are changing rapidly enough, the system with averaged rate constants is qualitatively similar, whereas when the rate constants are changing slowly enough (subject to some exceptions; see Example \ref{ex:disjoint species}), the correct approximation is by the constant rate processes rather than the average.

These results naturally open up the possibility of phase transitions, in terms of the stability of the model, as the parameter $\kappa$ varies. We capture this phenomenon in Examples \ref{ex:switching ergodic monomolecular} and \ref{ex:basic switching transient}.

The assumptions in our results match the conditions given in \cite{Benaim_LeBorgne_Malrieu_Zitt_2014} in the PDMP case, although that paper only addresses two of the four possibilities (namely, stability for small $\kappa$ and instability for large $\kappa$; see their Theorem 1.6). Consistently, in their examples they show that (to use our terminology) the system can be stable for small $\kappa$ but unstable for large $\kappa$, whereas we have examples both of this and of the reverse behavior (instability for small $\kappa$ paired with stability for large $\kappa$). There is a further difference with \cite{Benaim_LeBorgne_Malrieu_Zitt_2014} that is worth emphasizing, in addition to their notable restriction to dimension 2: their proof is based on a polar decomposition technique, which allows them to reduce the dimension of the problem by exploiting the fact that the dynamics of the angle is autonomous. This technique does not apply in the context of this paper. Instead, we prove our results from scratch using a novel Lyapunov function construction, which has the added benefit of working in higher dimensions. In fact, our technique could potentially be adapted back to their setting, though we do not address that possibility in this work. A Lyapunov function approach is also used in \cite{Pinsky_Scheutzow_1992}, however, the Lyapunov functions we construct in this paper rely on minor while indispensable changes caused by the switching environments, while their Lyapunov function constructions are based on a state space decomposition.

Above, we have discussed our results in the cases where $\kappa$ is sufficiently large or sufficiently small. We also study the case where $\kappa$ takes some intermediate value. It turns out that in that case, regardless of the behavior at the extremes, the process can undergo an arbitrarily large number of phase transitions between stable and unstable. While similar counterintuitive behavior has been noticed in \cite{Lawley_Mattingly_Reed_2014} in the context of PDMPs switching between linear ODEs, the models in that paper do not arise from CRNs. Moreover, though both here and in \cite{Lawley_Mattingly_Reed_2014} the key first step for showing arbitrarily many phase transitions is to construct a model which is stable for large and small $\kappa$ but is unstable for some intermediate $\kappa$, both the construction and verification are fundamentally different.

We would be remiss not to mention that we are not the first to look at specifically stochastic mass-action CRNs with stochastically varying coefficients. In \cite{Cappelletti_PalMajumder_Wiuf_2021}, they give conditions for such networks to be positive recurrent regardless of $\kappa$. The monomolecular networks from that paper are a subclass of the networks studied in this one, though unlike the present paper they allow for the stochastic environment to have an infinite state space. The major difference in the two papers is that the conditions for positive recurrence in \cite{Cappelletti_PalMajumder_Wiuf_2021} are structural, in contrast with the matrix conditions given here; this means that the results in \cite{Cappelletti_PalMajumder_Wiuf_2021} cannot capture stability transitions of the type studied in this paper.

Finally, it is worth highlighting that the models we study, under the mass-action and at-most-monomolecular assumptions, are ``continuous-time multitype branching processes with immigration in a random environment.'' While we are not aware of any results pertaining the stability of such general models, some related work exists for both discrete-time \cite{Key_1987,Z18} and continuous-time \cite{Kaplan_1973} processes.

In the remainder of this section, we will describe our notational conventions and definitions, as well as the precise setting in which we work. In Sections \ref{sec:Recurrence} and \ref{sec:Transience}, we provide sufficient conditions, in this setting, for models to be either exponentially ergodic or evanescent when $\kappa$ is either sufficiently large or sufficiently small. Section \ref{sec:Examples} provides various examples of applying these theorems, including Examples \ref{ex:switching ergodic monomolecular} and \ref{ex:basic switching transient} in which a stochastic mass-action CRN with stochastically varying rate constants transitions from stable to unstable or vice versa as $\kappa$ increases. Section \ref{sec:Intermediate Switching} deals with the case where $\kappa$ in the intermediate region where none of our previous theorems apply; we do not give a complete classification but show via several examples that the spectrum of possible behaviors is very rich. Section \ref{sec:Open Problems} addresses a possible generalization of Theorem \ref{thm:slow-switching-transience} which is left open by the present work. Lastly, the (already known) Lyapunov function theorems which undergird our proofs are provided in the Appendix, along with a summary of various results pertaining Metzler matrices which are useful for placing in context the hypotheses of our theorems.

\subsection{Notation and Terminology}\label{sec:Notation}

We use the abbreviation \emph{iff} in the standard manner, to mean ``if and only if". We denote the non-negative and positive integers by $\ZZ_{\ge0}$ and $\ZZ_{>0}$, respectively, and for $n\in\ZZ_{\ge0}$ we let $[n]:=\{1,\cdots,n\}$. We will denote the set of non-negative vectors in $\RR^d$ by $\RR^d_{\ge0}$ (that is, $v\in\RR^d_{\ge0}$ iff $v_i\ge0$ for all $i$), and similarly for the set of positive vectors $\RR_{>0}$, the non-positive vectors $\RR_{\le0}$, and the negative vectors $\RR_{<0}$.

If $w\in\RR^n$ is a vector $w=(w_1,\cdots,w_n)$, let $\operatorname{diag}(w)$ denote the $n\times n$ diagonal matrix whose $i$-th diagonal entry is $w_i$. If $v\in\RR^d$ is a vector, let $\norm v_{\ell^1}$ denote the $\ell^1$ norm of $v$ and let $\supp(v)\subseteq[d]$ denote the set of indices of nonzero entries of $v$; that is, $\supp(v):=\{m:v_m\ne0\}$.

We assume general knowledge of continuous time Markov chains (CTMCs) \cite{Norris_1997}. In particular, a Markov chain with unique stationary distribution $\pi$ is said to be \emph{exponentially ergodic} if there exists a constant $\eta>0$ and a function $f$ such that
\[
\norm{P_t(x,\cdot)-\pi(\cdot)}_{\mathrm{TV}}\le f(x) e^{-\eta t},
\]
for all $x$ and all $t$, where $P_t$ is as above and $\norm{\cdot}_{\mathrm{TV}}$ is the total variation norm. If $X(t)$ is a Markov chain on (a subset of) $\RR^d$ and $\norm{X(t)}_{\ell^1}\to\infty$ with positive probability when the process is started in state $x$, then we say $x$ is \emph{evanescent} for the process, consistent with the definition  \cite[Definition~3.1, p.526]{Meyn_Tweedie_1993} for Markov process on a general state space. Like transience and recurrence, evanescence is a class property of a Markov chain. Moreover, evanescence is equivalent to transience for irreducible CTMCs \cite[Section~3.1]{Meyn_Tweedie_1993}. In general, one may think of evanescent states as transient states from which the chain fails to reach a recurrent state with positive probability:

\begin{prop}\label{prop:evanescent}
    Let $X$ be a CTMC on $\mathbb S\subseteq \ZZ^d$, and let $x\in\mathbb S$. Then $x$ is evanescent for $X$ iff when started from state $x$ with positive probability the process never enters a recurrent state.
\begin{proof}[Proof Sketch]
    First, assume that $\lim\limits_{t\to\infty}\norm{X(t)}_{\ell^1}\to\infty$ occurs with positive probability. Given this event, the process almost surely does not enter a recurrent state because if a recurrent state is visited once then it is a.s.~visited infinitely many times \cite{Norris_1997}.
    
    For the reverse implication, suppose that $\liminf\limits_{t\to\infty}\norm{X(t)}_{\ell^1}$ is finite with probability one. Note that for all $R\in\RR$, the event that the process ever enters a recurrent state contains, up to a set of measure zero, the event that $\liminf\limits_{t\to\infty}\norm{X(t)}_{\ell^1}\le R$. Indeed, otherwise we could use the fact that $\{y\in\ZZ^d:\norm y_{\ell^1}\le R\}$ is finite to conclude that there existed a specific transient state $y$ which was visited infinitely often with positive probability, contradicting transience of $y$. Sending $R\to\infty$, the result follows from continuity of probability. 
\end{proof}
\end{prop}

The following is not a standard definition; we could not find it in the literature even thought it is based on standard concepts of Markov chain convergence.
\begin{definition}
 Let $X$ be a CTMC on $\mathbb S\subseteq \ZZ^d$. We say that $X$ \emph{converges exponentially fast} if
 \begin{enumerate}
     \item $X$ restricted to any closed communicating class is exponentially ergodic;
     \item there exists $\gamma\in\RR_{>0}$ such that for every $x\in S$ we have $\E_x[e^{\gamma \tau}]<\infty$, where $\tau=\inf\{t\in\RR_{>0}\,:\,X(t)$ is in a closed communicating class$\}$.
 \end{enumerate} 
\end{definition}
\begin{remark}
    By the Markov inequality, if $X$ converges exponentially fast we can conclude
    \begin{equation*}
        P_x(\tau>t)=P_x(e^{\gamma\tau}>e^{\gamma t})\leq \E_x[e^{\gamma\tau}]e^{-\gamma t},
    \end{equation*}
    so the probability of not being in a closed communicating class decreases exponentially fast, with a uniform rate $\gamma$ independent of the initial value $x$, and a multiplicative coefficient depending on $x$.
\end{remark}

The objects of study in this paper are \emph{chemical reactions networks}, or \emph{CRNs}. Let $\mathcal S$ be any finite set; we refer to the elements of $\mathcal S$ as \emph{species}. Linear combination of species with coefficients in $\ZZ_{\ge0}$ are called \emph{complexes}. Suppose we have a finite set $\mathcal R$ of \emph{reactions}, each of which is an arrow from one complex (called the \emph{source complex}) to a second complex (called the \emph{product complex}). A CRN is a tuple $(\mathcal S,\mathcal C,\mathcal R)$ where $\mathcal S$ and $\mathcal R$ are as discussed and $\mathcal C$ is the set of all complexes which appear in at least one reaction (either as a source complex or a product complex). CRNs are frequently drawn as a directed graph whose vertices are the complexes and whose edges are the reactions. While strictly speaking complexes are linear combinations over $\ZZ_{\ge0}$ of species, it is standard to view them as elements of $\ZZ_{\ge0}^{\mathcal S}$ by identifying linear combinations with their coefficients, and is even more common to abuse notation further by implicitly fixing an ordering of the species and viewing complexes as elements of $\ZZ_{\ge0}^d$ for $d\in\ZZ_{>0}$ the number of species.

One example CRN is $(\mathcal S_1,\mathcal C_1,\mathcal R_1)$ with species $\mathcal S_1=\{H_2O,H_2,O_2\}$, reaction $\mathcal R_1=\{2H_2+O_2\to 2H_2O\}$, and complexes $\mathcal C_1=\{2H_2+O_2,2H_2O\}$. Here the only source complex is $2H_2+O_2$ and the only product complex is $2H_2O$. We will say a CRN \emph{has at-most-monomolecular reactions} if for each $i$ and each $y\to y'\in\mathcal R$, the vector $y\in\ZZ_{\ge0}^d$ has at most one entry which is one and all other entries are zero. The example $(\mathcal S_1,\mathcal C_1,\mathcal R_1)$ from the beginning of this paragraph does not have at-most-monomolecular reactions because of the existence of a source complex $2H_2+O_2$, identified with the vector $(2,1,0)$; an example of an at-most-monomolecular reaction network would be one consisting only of the two reactions $0\to S_1\to 2S_2$, where $0$ denotes the empty linear combination $0S_1+0S_2$. Reactions with source complex $0$ are referred to as \emph{inflow reactions}.

A CRN $(\mathcal S,\mathcal C,\mathcal R)$ alone is not a dynamical system. There are multiple ways to get a dynamical system associated to a CRN; here we are concerned with specifically CTMC models. To each reaction $y\to y'\in\mathcal R$, assign a function, termed \emph{rate function}, $\lambda_{y\to y'}\colon \ZZ_{\ge0}^d\to\RR_{\geq 0}$. We consider a CTMC with state space $\ZZ_{\ge0}^d$ such that for each $v\in \ZZ^d$, the Markov chain can transition from state $x\in\ZZ_{\ge0}^d$ to state $x+v\in\ZZ_{\ge0}^d$ with transition rate
\[
\mathop{\sum_{y\to y'\in \mathcal R}}\limits_{y'-y=v}\lambda_{y\to y'}(x).
\]

We are particularly interested here in \emph{stochastic mass-action kinetics}: to each reaction $y\to y'\in\mathcal R$, assign a number, or \emph{rate constant}, $\kappa_{y\to y'}\in\RR_{>0}$, and define the rate functions as
$\lambda_{y\to y'}(x):=\kappa_{y\to y'}\prod_{m=1}^d\prod_{j=0}^{y_m-1}(x_m-j)$. Put another way, if the state space $x\in\ZZ_{\ge0}^d$ is interpreted as a vector listing the number of each species, then the rate of reaction $y\to y'$ is directly proportional to the number of ways to select the input species $y$ from the available species $x$, with constant of proportionality $\kappa_{y\to y'}$.

\subsection{Setting}

Throughout this paper we work in the following setting:

\begin{setting*}[SAL]
\makeatletter\def\@currentlabelname{SAL}\makeatother
\label{setting}
    Fix positive integers $n,d$. Let $M_1,\cdots,M_n$ be $d\times d$ matrices. For each $i=1,\cdots,n$, suppose that $M_i$ is associated to a CRN $X^i(t)$ with $d$ species; by ``associated", we mean that
    \begin{align}\label{eq:associated CRN}
        \sum_{y\to y'\in\mathcal R_i}\lambda_{i,y\to y'}(x)(y'_m-y_m)&=(M_ix)_m+o(\norm{x}_{\ell^1})
    \end{align}
    for each $i=1,\cdots,n$ and each $m=1,\cdots,d$ and each $x\in\ZZ^d_{\ge0}$, where $\mathcal R_i$ and $\lambda_{i,y\to y'}:\ZZ^d_{\ge0}\to[0,\infty)$ are the set and rate, respectively, of reactions of $X^i$, and $(M_ix)_m$ is the $m$-th entry of the vector $M_i x$.

    Suppose that $Q=(q_{ij})$ is a $n\times n$ transition rate matrix associated to an irreducible Markov chain on $[n]$, and that $w=(w_1,\cdots,w_n)$ is the unique stationary distribution of $Q$. Let $M=\sum w_iM_i$. Let $\kappa>0$, and notice that $\kappa Q$ is a transition rate matrix with the same stationary distribution as $Q$.

    We consider the Markov chain $Z_\kappa$ on $\ZZ^d_{\ge0}\times[n]$ whose last coordinate evolves according to $\kappa Q$, and whose first $d$ coordinates evolve according to $X^i$ when the last coordinate is in state $i$. More precisely, the only possible transitions are from a state of the form $(x,i)$ to a state of the form $(x',i)$ at rate
    \[
    \sum_{y\to y'\in\mathcal R_i\,:\,y'-y=x'-x}\lambda_{i,y\to y'}(x),
    \]
    and from a state of the form $(x,i)$ to a state of the form $(x,j)$ at rate $\kappa q_{ij}$, where $x, x'\in\ZZ^d_{\ge0}$ and $i,j\in[n]$. We refer to $Z_\kappa$ as the \emph{mixed process} associated to $\kappa Q$ and $X^1,\cdots,X^n$. We will denote by $\mathcal L_\kappa$ the generator of the Markov chain $Z_\kappa$. \hfill //
\end{setting*}

Equation \eqref{eq:associated CRN} says that in the $i$-th CRN, the (leading order) change of species $m$ as a function of the number $x_l$ of molecules of species $l$ is given by $x_l(M_i)_{m,l}$, where $(M_i)_{m,l}$ is the $m,l$-th entry of $M_i$. The $o(\norm{x}_{\ell^1})$ term allows for the existence of constant-rate inflow reactions.

The intended interpretation of the \nameref{setting} Setting  is that there are $n$ CRNs with a combined $d$ species, and which one of the CRNs is active at any given time is decided by a further Markov chain (the stochastic environment) on $[n]$ that evolves freely, with transition rates completely independent from the current species count. Then $w_i$ is the fraction of time which the CRN corresponding to $M_i$ is active, and so $M$ is the matrix giving the average rate of change of each species as a function of each other species.

In principle equation \eqref{eq:associated CRN} allows for the possibility of higher-order reactions whose rates cancel. When we prove theorems about transience we will need to restrict to the case where no such higher-order reactions exist. Our recurrence theorems, by contrast, will not have this restriction (see Example \ref{ex:switching ergodic non-monomolecular}).

Note that the \nameref{setting} Setting does not assume mass-action kinetics, and this is not needed to prove the recurrence theorems (Theorems \ref{thm:fast-switching-ergodicity} and \ref{thm:slow-switching-ergodicity}). In contrast, mass-action kinetics and at-most-monomolecular reactions are assumed to prove the transience theorems (Theorems \ref{thm:fast-switching-transience} and \ref{thm:slow-switching-transience}).

\section{Recurrence}\label{sec:Recurrence}

\subsection{Fast Switching}

Our first theorem gives a condition under which the mixed process converges exponentially fast when the environment transitions quickly enough. The condition, that there exists $v\in\RR^d_{>0}$ such that $vM\in\RR^d_{<0}$, implies by Corollary \ref{cor:one-environment-ergodicity} below that a CRN in the average environment $M$ would converge exponentially fast. Thus Theorem \ref{thm:slow-switching-ergodicity} gives a sense in which the switched model behaves like the averaged model when the switching rate is fast.

\begin{theorem}\label{thm:fast-switching-ergodicity}
Working in the \nameref{setting} Setting, suppose that there exists a vector $v\in\RR^d_{>0}$ such that $vM\in\RR^d_{<0}$. Then as long as $\kappa$ is sufficiently large, the mixed Markov chain $Z_\kappa$ converges exponentially fast.
\end{theorem}

\begin{remark}
    In light of Proposition \ref{prop:dec/inc direction}, the existence of a vector $v\in\RR^d_{>0}$ such that $vM\in\RR^d_{<0}$ is equivalent to $M$ being Hurwitz stable (recall that $M$ is Hurwitz stable if all eigenvalues of $M$ have negative real part; see section \ref{sec:lin-alg} for definitions and related discussion). \hfill //
\end{remark}

Before proceeding with the proof, we need a technical lemma which will also be useful for proving Theorem \ref{thm:fast-switching-transience}.

\begin{lemma}\label{lem:z^m}
In the \nameref{setting} Setting, for any $m\in[d]$ and any $v\in\RR^d$ there exists a vector $z^m=(z_1^m,z_2^m,\cdots,z_n^m)\in\RR^n_{>0}$ such that for each $i\in[n]$ we have
\begin{align*}
    \left(\sum_{j\ne i}q_{ij}(z^m_j-z^m_i)\right)+(vM_i)_m&=w_i^{-1}n^{-1}(vM)_m,
\end{align*}
where $(vM_i)_m$ and $(vM)_m$ denote the $m$-th entries of the vectors $vM_i$ and $vM$, respectively.

\begin{proof}
Fix $m$. Let $\psi=(\psi_1,\cdots,\psi_n)$ denote the vector whose $i$-th entry is the $m$-th entry of the (length $d$) vector $v(n^{-1}M-w_iM_i)$; that is,
\begin{align*}
    \psi_i:=(v(n^{-1}M-w_iM_i))_m.
\end{align*}
We claim that there exists a vector $z^m$ such that
\begin{align}\label{eq:diag(w)Qz=psi}
    \operatorname{diag}(w)Qz^m=\psi.
\end{align}
By the Fundamental Theorem of Linear Algebra, the image space of the matrix $\operatorname{diag}(w)Q$ is exactly the set of vectors orthogonal to the kernel of $(\operatorname{diag}(w)Q)^\intercal=Q^\intercal \operatorname{diag}(w)$. But since the Markov chain associated to $Q$ is irreducible, its stationary distribution $w$ is unique. It follows that, if $\widehat w$ is a vector with $\widehat wQ=0$, then $\widehat w$ is a scalar multiple of $w$.
But if $0=(\operatorname{diag}(w)Q)^\intercal\rho$ for some $\rho=(\rho_1,\cdots,\rho_n)^\intercal$ then
$
    0=\operatorname{diag}(w)\rho Q,
$
so in particular taking $\widehat w=\operatorname{diag}(w)\rho$ we see that $\rho$ must be a scalar multiple of the all-ones vector $\mathbf 1$. All this is to say that the kernel of $(\operatorname{diag}(w)Q)^\intercal$ is exactly the span of $\mathbf 1$, and therefore $\psi$ is in the image of $\operatorname{diag}(w)Q$ iff $\psi$ is orthogonal to $\mathbf 1$.

The scalar product between $\psi$ and $\mathbf 1$ is
\begin{align*}
    \psi\cdot\mathbf1
    &=\sum_{i=1}^n (v(n^{-1}M-w_iM_i))_m\\
    &=\left(\sum_{i=1}^nv(n^{-1}M-w_iM_i)\right)_m\\
    &=\left(v\left(\sum_{i=1}^n n^{-1}M-w_iM_i\right)\right)_m\\
    &=\left(v\left(M-M\right)\right)_m\\
    &=0.
\end{align*}
So $\psi$ is indeed orthogonal to $\mathbf 1$, and hence $\psi$ is the image of some vector $z^m$ under $\operatorname{diag}(w)Q$. Let $z^m$ be such. Notice that the $i$-th entry of $\operatorname{diag}(w)Q z^m$ is
\begin{align*}
    (\operatorname{diag}(w)Q z^m)_i
    &=\sum_{j=1}^n w_i q_{ij}z^m_j\\
    &=\left(\sum_{j\ne i} w_i q_{ij}z^m_j\right)+w_i q_{ii}z^m_i\\
    &=\left(\sum_{j\ne i} w_i q_{ij}z^m_j\right)-\sum_{j\ne i}w_i q_{ij} z^m_i\\
    &=\sum_{j\ne i} w_i q_{ij}(z^m_j-z^m_i).
\end{align*}
So equation \eqref{eq:diag(w)Qz=psi} tells us that for each $i$, we have
\begin{align*}
    \sum_{j\ne i} w_i q_{ij}(z^m_j-z^m_i)
    &=(v(n^{-1}M-w_iM_i))_m\\
    \sum_{j\ne i} w_i q_{ij}(z^m_j-z^m_i)+w_i(vM_i)_m&=n^{-1}(vM)_m\\
    \sum_{j\ne i}  q_{ij}(z^m_j-z^m_i)+(vM_i)_m&=w_i^{-1}n^{-1}(vM)_m.
\end{align*}
But this is exactly the equation we wanted $z^m$ to satisfy. \textit{A priori}, the vector $z^m$ need not have all positive entries. However, notice that if we add the same real number to every entry of $z^m$ then the equality above will be preserved, since the left-hand side is only a function of the differences $z^m_j-z^m_i$. Therefore, we may take every entry of $z^m$ to be positive without loss of generality.
\end{proof}
\end{lemma}

\begin{remark}
    The proof of Lemma \ref{lem:z^m} just given is not constructive, and we have not attempted to find a constructive proof in general. However, in the special case where $n$, the number of environments, is $2$, there is a simple construction worth mentioning. Specifically, in that case for $i=1,2$ we can define $z^m_i$ via $2w_1w_2 z^m_i=(w_ivM_i)_m$. \hfill //
\end{remark}

\begin{proof}[Proof of Theorem \ref{thm:fast-switching-ergodicity}]
Define $h_\kappa:\ZZ^d_{\ge0}\times[n]\to\RR_{\ge0}$ via $h_\kappa(x,i) = (v+\kappa^{-1} u^i)\cdot x$ for some vector $u^i\in\mathbb{R}^d_{\ge0}$ to be picked later. Then
\begin{align*}
    \mathcal L_\kappa h_\kappa(x,i)
    &=\sum_{j\ne i}\kappa q_{ij}\kappa^{-1}(u^j-u^i)\cdot x
    +\sum_{y\to y'\in\mathcal R_i}\lambda_{i,y\to y'}(x)(v+\kappa^{-1}u^i)\cdot (y'-y)\\
    &=\sum_{j\ne i} q_{ij}(u^j-u^i)\cdot x
    +\sum_{m=1}^d (v_m+\kappa^{-1}u^i_m) \sum_{y\to y'\in\mathcal R_i}\lambda_{i,y\to y'}(x)(y'_m-y_m)\\
    &=\sum_{j\ne i}q_{ij}(u^j-u^i)\cdot x
    +\sum_{m=1}^d (v_m+\kappa^{-1}u^i_m)(M_ix)_m+o(\norm x_{\ell^1})\\
    &=\sum_{j\ne i}q_{ij}(u^j-u^i)\cdot x
    +(v+\kappa^{-1}u^i)\cdot (M_ix)+o(\norm x_{\ell^1})\\
    &=\left(\sum_{j\ne i} q_{ij}(u^j-u^i)
    +(v+\kappa^{-1}u^i)M_i\right)\cdot x+o(\norm x_{\ell^1})
\end{align*}
By Theorem \ref{thm:Lyapunov-exponential-ergodicity}, it suffices to show for each $i\in[n]$ the existence of $u^i\in\mathbb{R}^d_{\ge0}$  such that 
\begin{equation*}
\sum_{j\neq i} q_{ij} (u^j - u^i) + (v+\kappa^{-1}u^i)M_i <0
\end{equation*}
But by Lemma \ref{lem:z^m}, there exists $u^i\in\mathbb R^d_{>0}$ such that 
\[
    \sum_{j\neq i} q_{ij} (u^j - u^i) + vM_i = w_i^{-1}n^{-1}vM;
\]
specifically, take $u^i_m=z^m_i$ where $z^m\in\RR^n_{>0}$ is the vector given by Lemma \ref{lem:z^m} for $m\in[d]$. Hence we obtain exponential ergodicity of the process randomly switching among $n$ environments for large $\kappa$.
\end{proof}

\subsection{Slow Switching}

While the previous theorem shows that the switched model inherits stability from the averaged model when the switching rate is large enough, this next theorem shows that the switched model inherits stability from each individual environment when the switching rate is slow.

\begin{theorem}\label{thm:slow-switching-ergodicity}
    Working in the \nameref{setting} Setting, suppose that for each $i=1,\cdots,n$ there exists a vector $v^i\in\RR^d_{>0}$ such that $v^iM_i\in\RR^d_{<0}$. Then as long as $\kappa$ is sufficiently small, the mixed Markov chain $Z_\kappa$ converges exponentially fast.
\end{theorem}

\begin{proof}[Proof of Theorem \ref{thm:slow-switching-ergodicity}]
Define $h:\ZZ^d_{\ge0}\times[n]\to\RR_{\ge0}$ via $h(x,i)=v^i\cdot x$. Then
\begin{align*}
    \mathcal L_\kappa h(x,i)
    &=\sum_{j\ne i}\kappa q_{ij}(v^j-v^i)\cdot x
    +\sum_{y\to y'\in\mathcal R_i}\lambda_{i,y\to y'}(x) (v^i\cdot (y'-y))\\
    &=\sum_{j\ne i}\kappa q_{ij}(v^j-v^i)\cdot x
    +\sum_{m=1}^dv^i_m\sum_{y\to y'\in\mathcal R_i}\lambda_{i,y\to y'}(x) (y'_m-y_m)\\
    &=\sum_{j\ne i}\kappa q_{ij}(v^j-v^i)\cdot x
    +\sum_{m=1}^dv^i_m(M_ix)_m+o(\norm x_{\ell^1})\\
    &=\sum_{j\ne i}\kappa q_{ij}(v^j-v^i)\cdot x
    +v^i\cdot (M_ix)+o(\norm x_{\ell^1})\\
    &=\left(\sum_{j\ne i}\kappa q_{ij}(v^j-v^i)+v^iM_i\right)\cdot x+o(\norm x_{\ell^1}).
\end{align*}
By Theorem \ref{thm:Lyapunov-exponential-ergodicity}, it suffices to show that
\begin{align*}
    \sum_{j\ne i}\kappa q_{ij}(v^j-v^i)+v^iM_i<0
\end{align*}
for each $i$. But $v^iM_i<0$ by assumption, and then rest of the expression is converging to zero as $\kappa\to0$. We conclude that the Markov chain is exponentially ergodic for sufficiently small $\kappa$, as claimed.
\end{proof}

\subsection{One environment}

The special case of the previous two theorems in which there is only one environment is worth being stated in its own right.

\begin{cor}\label{cor:one-environment-ergodicity}
Suppose that $M$ is a matrix which is associated to a CRN in the sense of equation \eqref{eq:associated CRN}. If there exists a vector $v\in\RR_{>0}^d$ such that $vM\in\RR_{<0}^d$, then the CRN converges exponentially fast.
\begin{proof}
    We can view this as a switched process in one environment with $M_1=M$ and switching matrix $Q=(0)$. So if we consider the process with matrix $\kappa Q$, then Theorem \ref{thm:slow-switching-ergodicity} says that the process converges exponentially fast if $\kappa$ is small enough and Theorem \ref{thm:fast-switching-ergodicity} says that the process converges exponentially fast if $\kappa$ is large enough. But $\kappa Q=Q$ in this case, so either theorem is enough to conclude the claimed result.
\end{proof}
\end{cor}

\section{Transience}\label{sec:Transience}

\subsection{Fast Switching}

Recall that Theorem \ref{thm:fast-switching-ergodicity} showed that the switched process is stable when the average environment is stable and the switching rate is fast enough. The next theorem is an analog for instability. Together, they show that for fast enough switching rates, the switched model will behave like the averaged model.

\begin{theorem}\label{thm:fast-switching-transience}
In the \nameref{setting} Setting, suppose additionally that each CRN is endowed with mass-action kinetics and has at-most-monomolecular reactions. Suppose that there exists a vector $v\in\RR^d_{\ge 0}$ such that $(vM)_m>0$ for each $m\in\supp(v)$, where $(vM)_m$ is the $m$-th entry of the vector $vM$. Then as long as $\kappa$ is sufficiently large, we have that $(x,i)$ is an evanescent state for the mixed Markov chain $Z_\kappa$ for every $i\in[n]$ and every $x\in\ZZ^d_{\ge0}$ such that $\supp(v)\cap\supp(x)\ne\emptyset$.
\end{theorem}

\begin{remark}
    The condition $\supp(v)\cap\supp(x)\ne\emptyset$ means that $x$ is non-zero in at least one coordinate in which $v$ is non-zero. As for the transience or recurrence of the states where $x$ is zero in every coordinate in which $v$ is non-zero, there are two cases. The first is that, given such an $x$, some state $z$ with $\supp(v)\cap\supp(z)\ne\emptyset$ is reachable from $x$. In that case $z$ is transient and hence so is $x$. (See Example~\ref{ex:basic switching transient} for an example of such reasoning.) The second case is that, given $x$, no such $z$ is reachable. In that case the problem reduces to studying a lower-dimensional model (that is, one with fewer species) for which $x$ may or may not be transient, and potential options for studying this include Theorems \ref{thm:fast-switching-ergodicity} or \ref{thm:fast-switching-transience} from this paper.\hfill //
\end{remark}

\begin{remark}
    The condition that there exists a vector $v$ such that $(vM)_m>0$ for each $m\in\supp(v)$ should be thought of as an instability condition on the matrix $M$. Indeed, if $v=0$ then Theorem \ref{thm:fast-switching-transience} does not give transience for any states $(x,i)$, whereas by Proposition \ref{prop:Hurwitz-unstable} a non-zero $v$ such that $(vM)_m>0$ for each $m\in\supp(v)$ exists iff $M$ is Hurwitz-unstable (that is, iff $M$ has an eigenvalue with strictly positive real part). The conclusion of the theorem, then, is that if the mixed matrix $M$ is unstable in a linear algebra sense, then the mixed stochastic process is unstable in the usual Markov chain sense as long as the environment switches rapidly enough. \hfill //
\end{remark}

The idea of the proof of Theorem \ref{thm:fast-switching-transience} will be to define a Lyapunov function using Lemma \ref{lem:z^m}, show that the generator applied to this function is positive as long as the function is sufficiently large, and then apply Theorem \ref{thm:Lyapunov-transience} to conclude that the claimed states are evanescent. For this final step it will be convenient if the only reactions $y\to y'$ which are invisible to $v$ in the sense that $v\cdot (y'-y)=0$ are those for which $\supp(v)\cap\supp(y'-y)=\emptyset$. As it turns out, we can assume this without loss of generality, which follows from the next lemma.

\begin{lemma}\label{lem:v not orthogonal to reactions}
In the setting of Theorem \ref{thm:fast-switching-transience}, suppose we are given a finite collection $\Xi\subset\RR^d$ with the property that for all $\xi\in\Xi$ we have $\supp(v)\cap\supp(\xi)\ne\emptyset$. Then without loss of generality we may assume that $v\cdot\xi\ne0$ for all $\xi\in\Xi$.

Explicitly, the claim is that in the setting of Theorem \ref{thm:fast-switching-transience} there exists a vector $v^\varepsilon$ such that:
\begin{itemize}
    \item $v^\varepsilon\in\RR^d_{\ge0}$.
    
    \item $(v^\varepsilon M)_m>0$ for each $m$ such that $v^\varepsilon_m>0$.

    \item $\supp(v^\varepsilon)=\supp(v)$ (and hence in particular $\supp(v^\varepsilon)\cap\supp(x)=\supp(v)\cap\supp(x)$ for all $x\in\ZZ^d_{\ge0}$).

    \item $v^\varepsilon\cdot\xi\ne0$ for all $\xi\in\Xi$.
\end{itemize}
\begin{proof}
Let $v\in\RR^d_{\ge0}$ be the vector given by the statement of the theorem, and consider the subspace $A$ of $\RR^d$ defined by
\[
    A:=\{z\in\RR^d:\supp(z)\subseteq\supp(v)\}=\operatorname{span}\{e^m:v_m>0\},
\]
where $e^m$ is the unit vector in direction $m$. For each $\xi\in\Xi$, the space $\{z\in A:z\cdot\xi=0\}$ is a subspace of $A$ of dimension at most one less than that of $A$. Since the union of finitely many lower-dimensional sets is not the whole space, there exists some $z\in\RR^d$ such that $z\cdot\xi\ne0$ for all $\xi\in\Xi$ and also $\supp(z)\subseteq\supp(v)$.

Let $z$ be such. For $\varepsilon>0$, consider the vector $v^\varepsilon:=v+\varepsilon z$. Notice that as long as $\varepsilon$ is small enough, $v^\varepsilon$ will be a non-negative vector with $\supp(v^\varepsilon)=\supp(v)$ (here it is important that $v$ is non-negative and $\supp(z)\subseteq\supp(v)$). Furthermore, notice that
\begin{align*}
    (v^\varepsilon M)_m
    &=(vM)_m+\varepsilon (zM)_m.
\end{align*}
From this it follows that if $\varepsilon$ is small enough, $(v^\varepsilon M)_m$ will be positive whenever $(vM)_m$ is.

Lastly, we claim that $v^\varepsilon\cdot\xi\ne0$ for all $\xi\in\Xi$ as long as $\varepsilon$ is small enough. Toward this end, notice that $v^\varepsilon\cdot\xi=v\cdot\xi+\varepsilon z\cdot\xi$; since $z\cdot\xi\ne0$ this quantity is a linear function of $\varepsilon$ (with non-zero slope). It follows that for each $\xi$ there is exactly one value of $\varepsilon$ for which $v^\varepsilon\cdot\xi=0$. The desired result follows.
\end{proof}
\end{lemma}

Because a partial result from the proof of Theorem \ref{thm:fast-switching-transience} will be useful for us later, most of the proof of the theorem will be outsourced in the form of the next two Lemmas, Lemma \ref{lem:generator positive} and Lemma \ref{lem:v cdot x^0>C}. (Specifically, the proof of Theorem \ref{thm:slow-switching-transience} will use Lemma \ref{lem:v cdot x^0>C} via Corollary \ref{cor:one-environment-lemma}.)

\begin{lemma}\label{lem:generator positive}
    In the setting of Theorem \ref{thm:fast-switching-transience}, let $I:=\supp(v)$ denote the set of species which are visible to $v$. For $m\in I$, let $z^m$ be the vector given to us by applying Lemma \ref{lem:z^m} to $v$; for $m\notin I$, let $z^m$ be the vector of all zeros. For $x\in\ZZ^d_{\ge0}$ and $i\in\{1,2,\cdots,n\}$, let
    \begin{align*}
        h_\kappa(x,i)=1-\frac{1}{1+\sum\limits_{m=1}^d (v_m +\kappa^{-1} z^m_i)x_m}.
    \end{align*}
    Let $\mathcal L_\kappa$ denote the generator of the Markov chain $Z_\kappa$. Then for each large enough $\kappa$, there exists some constant $b$ (depending implicitly on $\kappa$ but not on $x$ or $i$) such that $\mathcal L_\kappa h_\kappa(x,i)>0$ whenever $v\cdot x\ge b$.
\begin{proof}
    For $i=1,\cdots,n$, let $\widetilde{\mathcal R}_i$ denote the subset of $\mathcal R_i$ consisting of reactions which do not take the $m$-th species as an input for any $m\notin I$. In words, $\widetilde{\mathcal R}_i$ is the set of reactions $y\to y'$ such that every species which appears in $y$ is visible to $v$.

We claim that if $y\to y'\notin\widetilde{\mathcal R}_i$ then $h_\kappa(x-y+y',i)\ge h_\kappa(x,i)$. Indeed, suppose that $y\to y'\notin\widetilde{\mathcal R}_i$. If $m\notin I$ then $v_m+\kappa^{-1}z^m_i=0$ and if $m\in I$ we have $y'_m-y_m\ge0$ (it is important here that $y$ is at-most-monomolecular). It follows that each term in the sum in the denominator of $h_\kappa(x-y+y',i)$ is no smaller than the corresponding term in $h_\kappa(x,i)$, so $h_\kappa(x-y+y',i)\ge h_\kappa(x,i)$ as claimed.

Notice that for each $x$ and each $i$,
\begin{align*}
    &\mathcal L_\kappa h_\kappa(x,i)\\
    &=\sum_{j\ne i}\kappa q_{ij}(h_\kappa(x,j)-h_\kappa(x,i))
    +\sum_{y\to y'\in\mathcal R_i}\lambda_{i,y\to y'}(x)(h_\kappa(x-y+y',i)-h_\kappa(x,i))\\
    &\ge\sum_{j\ne i}\kappa q_{ij}(h_\kappa(x,j)-h_\kappa(x,i))
    +\sum_{y\to y'\in\widetilde{\mathcal R}_i}\lambda_{i,y\to y'}(x)(h_\kappa(x-y+y',i)-h_\kappa(x,i))\\
    &=\frac{1}{1+\sum\limits_{m=1}^d (v_m +\kappa^{-1} z^m_i)x_m}\left(\sum_{j\ne i}\frac{\kappa q_{ij}\sum\limits_{m=1}^d\kappa^{-1}(z_j^m-z_i^m)x_m}{1+\sum\limits_{m=1}^d (v_m +\kappa^{-1} z^m_j)x_m}
    +\sum_{y\to y'\in\widetilde{\mathcal R}_i}\frac{\lambda_{i,y\to y'}(x)\sum\limits_{m=1}^d (v_m +\kappa^{-1} z^m_i)(y'_m-y_m)}{1+\sum\limits_{m=1}^d (v_m +\kappa^{-1} z^m_i)(x_m-y_m+y'_m)}\right)\\
    &=:g_\kappa(x,i)
\end{align*}
where the inequality follows from the discussion in the previous paragraph, and the last line is the definition of $g_\kappa$.

Consider the quantity
\begin{align*}
    *:=\sum_{j\ne i}\frac{1+\sum\limits_{m=1}^d (v_m +\kappa^{-1} z^m_i)x_m}{1+\sum\limits_{m=1}^d (v_m +\kappa^{-1} z^m_j)x_m}\left( q_{ij}\sum\limits_{m=1}^d(z_j^m-z_i^m)x_m\right)
    +\sum_{y\to y'\in\widetilde{\mathcal R}_i}\lambda_{i,y\to y'}(x)\sum_{m=1}^d (v_m +\kappa^{-1} z^m_i)(y'_m-y_m).
\end{align*}
The difference between $*$ and $(1+\sum_{m=1}^d (v_m +\kappa^{-1} z^m_i)x_m)^2 g_\kappa(x,i)$ is 
\begin{align*}
    *\,-\,&\left(1+\sum_{m=1}^d (v_m +\kappa^{-1} z^m_i)x_m\right)^2g_\kappa(x,i)\\
    &=\sum_{y\to y'\in\widetilde{\mathcal R}_i}\sum_{m=1}^d\lambda_{i,y\to y'}(x) (v_m +\kappa^{-1} z^m_i)(y'_m-y_m)\left(1-\frac{1+\sum\limits_{m=1}^d (v_m +\kappa^{-1} z^m_i)x_m}{1+\sum\limits_{m=1}^d (v_m +\kappa^{-1} z^m_i)(x_m-y_m+y'_m)}\right)\\
    &=\sum_{y\to y'\in\widetilde{\mathcal R}_i}\sum_{m=1}^d\lambda_{i,y\to y'}(x) (v_m +\kappa^{-1} z^m_i)(y'_m-y_m)\left(\frac{\sum\limits_{m=1}^d (v_m +\kappa^{-1} z^m_i)(y'_m-y_m)}{1+\sum\limits_{m=1}^d (v_m +\kappa^{-1} z^m_i)(x_m-y_m+y'_m)}\right).
\end{align*}
The term $\lambda_{i,y\to y'}(x) (v_m +\kappa^{-1} z^m_i)(y'_m-y_m)$ only depends on $x$ via $\lambda_{i,y\to y'}(x)$, and since the mass-action CRN associated to $M_i$ has at-most-monomolecular reactions we know that $\lambda_{i,y\to y'}(x)$ is at most linear in $x$. Moreover, by definition of $\widetilde{\mathcal R}_i$ we know that the coefficient of $x_m$ is zero for each $m\notin I$.

Meanwhile, the numerator of the fraction in the last line does not depend on $x$, and the denominator is (affine) linear in $x$ with each $x_m$ having a non-negative coefficient which is positive iff $m\in I$. Therefore, each term in the (finite!) sum above is bounded as $x_m$ varies, and so the difference between $*$ and $(1+\sum_{m=1}^d (v_m +\kappa^{-1} z^m_i)x_m)^2g_\kappa(x,i)$ is bounded. We conclude that if we are able to show that $*\to\infty$ as $v\cdot x\to\infty$, then it would follow that $(1+\sum_{m=1}^d (v_m +\kappa^{-1} z^m_i)x_m)^2\mathcal L_\kappa h_\kappa(x,i)\to\infty$ as well.

Now we turn to analyzing $*$. Notice that
\begin{align*}
    \min\left\{1,\min_{l\in I}\left\{\frac{v_l +\kappa^{-1} z^l_i}{v_l +\kappa^{-1} z^l_j}\right\}\right\}\left(1+\sum_{m=1}^d (v_m +\kappa^{-1} z^m_j)x_m\right)
    &\le 1+\sum_{m=1}^d (v_m +\kappa^{-1} z^m_i)x_m\\
    \min\left\{1,\min_{l\in I}\left\{\frac{v_l +\kappa^{-1} z^l_i}{v_l +\kappa^{-1} z^l_j}\right\}\right\}
    &\le \frac{1+\sum\limits_{m=1}^d (v_m +\kappa^{-1} z^m_i)x_m}{1+\sum\limits_{m=1}^d (v_m +\kappa^{-1} z^m_j)x_m}
\end{align*}
and similarly
\begin{align*}
    \frac{1+\sum\limits_{m=1}^d (v_m +\kappa^{-1} z^m_i)x_m}{1+\sum\limits_{m=1}^d (v_m +\kappa^{-1} z^m_j)x_m}
    &\le \max\left\{1,\max_{l\in I}\left\{\frac{v_l +\kappa^{-1} z^l_i}{v_l +\kappa^{-1} z^l_j}\right\}\right\}
\end{align*}
Therefore, if we define
\begin{align*}
    a^m_{i,j}=\begin{cases}\displaystyle
        \min\left\{1,\min_{l\in I}\left\{\frac{v_l +\kappa^{-1} z^l_i}{v_l +\kappa^{-1} z^l_j}\right\}\right\}      &     \qquad z_j^m-z_i^m>0\\
        1   &     \qquad z_j^m-z_i^m=0\\\displaystyle
        \max\left\{1,\max_{l\in I}\left\{\frac{v_l +\kappa^{-1} z^l_i}{v_l +\kappa^{-1} z^l_j}\right\}\right\}      &     \qquad z_j^m-z_i^m<0
    \end{cases},
\end{align*}
then
\begin{align*}
    *&=\sum_{j\ne i}\frac{1+\sum\limits_{m=1}^d (v_m +\kappa^{-1} z^m_i)x_m}{1+\sum\limits_{m=1}^d (v_m +\kappa^{-1} z^m_j)x_m}\left( q_{ij}\sum_{m=1}^d(z_j^m-z_i^m)x_m\right)
    +\sum_{y\to y'\in\widetilde{\mathcal R}_i}\lambda_{i,y\to y'}(x)\sum_{m=1}^d (v_m +\kappa^{-1} z^m_i)(y'_m-y_m)\\
    &\ge\sum_{j\ne i}q_{ij}\sum_{m=1}^d a^m_{i,j}(z_j^m-z_i^m)x_m
    +\sum_{y\to y'\in\widetilde{\mathcal R}_i}\lambda_{i,y\to y'}(x)\sum_{m=1}^d (v_m +\kappa^{-1} z^m_i)(y'_m-y_m)
\end{align*}

Define $\varphi^i$ to be the vector $(\varphi^i_1,\cdots,\varphi^i_d)$ whose $m$-th entry is $\varphi^i_m:=v_m +\kappa^{-1} z^m_i$. If the sum over $\widetilde{\mathcal R}_i$ in the last line above were over $\mathcal R_i$ instead, then it would be
\begin{align*}
    \sum_{y\to y'\in\mathcal R_i}\lambda_{i,y\to y'}(x)\sum_{m=1}^d (v_m +\kappa^{-1} z^m_i)(y'_m-y_m)
    &=\sum_{m=1}^d \varphi^i_m(M_ix)_m+o(\norm{x}_{\ell^1})\\
    &=\varphi^i\cdot(M_ix)+o(\norm{x}_{\ell^1})\\
    &=(\varphi^i M_i)\cdot x+o(\norm{x}_{\ell^1})\\
    &=\sum_{m=1}^d(\varphi^i M_i)_mx_m+o(\norm{x}_{\ell^1}).
\end{align*}
But the distinction between $\widetilde{\mathcal R}_i$ and $\mathcal R_i$ is exactly that the former leaves out the reactions whose rate is a multiple of $x_m$ for some $m\notin I$. Therefore,
\begin{align*}
    \sum_{y\to y'\in\widetilde{\mathcal R}_i}\lambda_{i,y\to y'}(x)\sum_{m=1}^d (v_m +\kappa^{-1} z^m_i)(y'_m-y_m)
    &=\sum_{m\in I}(\varphi^i M_i)_mx_m+o(\norm{x}_{\ell^1}).
\end{align*}
Plugging this into the above and applying the fact that $z_j^m=0$ for each $m\notin I$ and each $j$, we get
\begin{align*}
    *&\ge\sum_{j\ne i}q_{ij}\sum_{m=1}^d a^m_{i,j}(z_j^m-z_i^m)x_m
    +\sum_{y\to y'\in\widetilde{\mathcal R}_i}\lambda_{i,y\to y'}(x)\sum_{m=1}^d (v_m +\kappa^{-1} z^m_i)(y'_m-y_m)\\
    &=\sum_{j\ne i}q_{ij}\sum_{m=1}^d a^m_{i,j}(z_j^m-z_i^m)x_m
    +\sum_{m\in I}(\varphi^i M_i)_mx_m+o(\norm{x}_{\ell^1})\\
    &=\sum_{j\ne i}q_{ij}\sum_{m\in I} a^m_{i,j}(z_j^m-z_i^m)x_m
    +\sum_{m\in I}(\varphi^i M_i)_mx_m+o(\norm{x}_{\ell^1})\\
    &=\sum_{m\in I}\left(\sum_{j\ne i}q_{ij} a^m_{i,j}(z_j^m-z_i^m)+(\varphi^i M_i)_m\right)x_m+o(\norm{x}_{\ell^1}).
\end{align*}
Written in this form, it is clear that to show $*\to\infty$ as $v\cdot x\to\infty$, it suffices to show that
\begin{align}\label{eq:leading coefficient}
    \sum_{j\ne i}q_{ij}a^m_{i,j}(z_j^m-z_i^m)
    +(\varphi^i M_i)_m>0
\end{align}
for each $m\in I$. But notice that $a^m_{i,j}\to1$ as $\kappa\to\infty$ and $\varphi_m^i\to v_m$ as $\kappa\to\infty$, and the expression above does not otherwise depend on $\kappa$. So
\begin{align*}
    \lim_{\kappa\to\infty}\sum_{j\ne i}q_{ij}a^m_{i,j}(z_j^m-z_i^m)+(\varphi^i M_i)_m
    &=\sum_{j\ne i}q_{ij}(z_j^m-z_i^m)+(v M_i)_m.
\end{align*}
Recall that the vector $z^m$ was selected using Lemma \ref{lem:z^m} exactly so that the right-hand side of this equation would be equal to $w_i^{-1}n^{-1}(vM)_m$. Since $(vM)_m>0$ for $m\in I$ by assumption, we conclude that \eqref{eq:leading coefficient} is satisfied as long as $\kappa$ is sufficiently large. Fix $\kappa$ large enough that the inequality \eqref{eq:leading coefficient} is satisfied. Then it follows that we have $*\to\infty$ as $v\cdot x\to\infty$ and hence $(1+\sum_{m=1}^d (v_m +\kappa^{-1} z^m_i)x_m)^2\mathcal L_\kappa h_\kappa(x,i)\to\infty$ as well. Fix $b$ large enough that when $v\cdot x\ge b$, we have $(1+\sum_{m=1}^d (v_m +\kappa^{-1} z^m_i)x_m)^2\mathcal L_\kappa h_\kappa(x,i)>0$. Then because the square is positive, for all such $x$ we have $\mathcal L_\kappa h_\kappa(x,i)>0$ as well. (In principal $b$ should depend on $i$ as well as on $\kappa$, but there are only finitely many possibilities for $i$ so we're just picking $b$ large enough that it works for all $i$). In other words, $\mathcal L_\kappa h_\kappa(x,i)>0$ whenever $v\cdot x\ge b$, which is exactly the result we set out to prove.
\end{proof}
\end{lemma}

\begin{lemma}\label{lem:v cdot x^0>C}
    In the setting of Theorem \ref{thm:fast-switching-transience}, for any constant $C\in\RR$, any $i\in[n]$, and any $x\in\ZZ^d_{\ge0}$ such that $\supp(v)\cap\supp(x)\ne\emptyset$, there exists an $x^0\in\ZZ^d_{\ge0}$ such that $(x^0,i)$ is reachable from $(x,i)$ and $v\cdot x^0>C$.
\begin{proof}
Let $\Xi=\{y'-y:\exists i, y\to y'\in\mathcal R_i\text{ and }\supp(v)\cap\supp(y'-y)\ne\emptyset\}$ be the set of reactions in any environment which affect at least one species seen by $v$. By Lemma \ref{lem:v not orthogonal to reactions}, we may assume without loss of generality that $v$ sees all these reactions; that is, that $v\cdot(y'-y)\ne0$ for all $y'-y\in\Xi$.

For $m\in[d]$ let $z^m$ be as in the statement of Lemma \ref{lem:generator positive}; similarly, let $h_\kappa$ and $\mathcal L_\kappa$ be as in the statement of Lemma \ref{lem:generator positive}. As in the proof of Lemma \ref{lem:generator positive}, let $\varphi^i=(\varphi^i_1,\cdots,\varphi^i_d)$ be the vector defined by $\varphi^i_m:=v_m +\kappa^{-1} z^m_i$.

The proof will consist of three claims.

\textit{Claim 1:} As long as $\kappa$ is large enough, for any $i$ and any reaction $y\to y'\in\mathcal R_i$, the numbers $\varphi^i\cdot(y'-y)$ and $v\cdot(y'-y)$ have the same sign (positive, negative, or zero).

\vspace{.2cm}

Notice that for any reaction $y\to y'$, 
\begin{align*}
    \varphi^i\cdot(y'-y)
    &=\sum_{m=1}^d(v_m+\kappa^{-1}z^m_i)(y'_m-y_m)\\
    &=v\cdot(y'-y)+\kappa^{-1}\sum_{m=1}^d(z^m_i)(y'_m-y_m)
\end{align*}
In particular, $\varphi^i\cdot(y'-y)\to v\cdot(y'-y)$ as $\kappa\to\infty$, and hence $\varphi^i\cdot(y'-y)$ will have the same sign as $v\cdot(y'-y)$ whenever $\kappa$ is large enough, provided $v\cdot(y'-y)\ne0$. By our application of Lemma \ref{lem:v not orthogonal to reactions} above, this includes all reactions $y\to y'$ with $\supp(v)\cap\supp(y'-y)\ne\emptyset$. Since there are only finitely many possible reactions, as long as $\kappa$ is large enough it follows that $\varphi^i\cdot(y'-y)$ has the same sign as $v\cdot(y'-y)$ whenever $\supp(v)\cap\supp(y'-y)\ne\emptyset$.

Now recall that $z_i^m=0$ whenever $v_m=0$. Therefore, $\supp(\varphi^i)=\supp(v)$ and hence if $y\to y'$ is a reaction such that $\supp(v)\cap\supp(y'-y)=\emptyset$ then $\supp(\varphi^i)\cap\supp(y'-y)=\supp(v)\cap(y'-y)=\emptyset$ and in particular $\varphi^i\cdot(y'-y)=0=v\cdot(y'-y)$. Therefore, $\varphi^i\cdot(y'-y)$ and $v\cdot(y'-y)$ still have the same sign in the case where $\supp(v)\cap\supp(y'-y)=\emptyset$, completing the proof of the claim.

\vspace{.2cm}

\textit{Claim 2:} If $x$ is such that $\supp(v)\cap\supp(x)\ne\emptyset$, then there exists at least one environment $i$ and at least one reaction $y\to y'\in\mathcal R_i$ such that $\lambda_{i,y\to y'}(x)>0$ and $v\cdot(y'-y)>0$.

\vspace{.2cm}

Let $x$ be such that $\supp(v)\cap\supp(x)\ne\emptyset$. For $m=1,\cdots,d$, let $e^m$ denote the standard basis vector in direction $m$ (that is, $e^m_l$ is one if $m=l$ and zero otherwise). Fix $m\in\supp(v)\cap\supp(x)$ and consider the vector $z:=x+e_m \ceil{(v_m)^{-1} b}$, where $\ceil{(v_m)^{-1} b}$ is the least integer greater than $(v_m)^{-1} b$ and $b$ the constant given by Lemma \ref{lem:generator positive}. Notice that $v\cdot z\ge v\cdot x+b>b$. It follows from Lemma \ref{lem:generator positive} that $\mathcal L_\kappa h_\kappa(z,j)>0$ for every $j$, and hence for each $j$ there exists at least one transition which increases $h_\kappa$ in state $(z,j)$. In particular, there exists at least one transition in state $(z,i)$ which increases $h_\kappa$, where $i$ is selected such that $h_\kappa(z,i)=\max_jh_\kappa(z,j)$. But this transition can't be a change of environment by choice of $i$, so it has to be a reaction $y\to y'\in\mathcal R_i$. But notice by definition of $h_\kappa$ that $h_\kappa(z+y'-y,i)>h_\kappa(z,i)$ implies that $\varphi^i\cdot(y'-y)>0$. It follows from Claim 1 that $v\cdot(y'-y)>0$.

It remains only to show that $\lambda_{i,y\to y'}(x)>0$. Indeed, because all reactions have at-most-monomolecular inputs, the only way that $\lambda_{i,y\to y'}(x)$ could be zero was if $x_l=0$ for the (necessarily unique) species such that $y_l=1$. Notice that if $l\ne m$ then $x_l=z_l$ and so the fact that $y\to y'$ is admissible in state $z$ implies that it is admissible in state $x$, and if $l=m$ then $x_m>0$ by choice of $m$. So in either case $\lambda_{i,y\to y'}(x)>0$, as claimed.

\vspace{.2cm}

\textit{Claim 3:} For any constant $C\in\RR$ and any $i\in[n]$, if $x\in\ZZ_{\ge0}^d$ is such that $\supp(v)\cap\supp(x)\ne\emptyset$, there exists a $x^0\in\ZZ_{\ge0}^d$ such that $(x^0,i)$ is reachable from $(x,i)$ and $v\cdot x^0>C$.

\vspace{.2cm}

We first note that if $x$ is a state to which we can apply Claim 2, then Claim 2 can also be applied to the resulting state $x+y'-y$. Indeed, if $\supp(x)\cap\supp(v)\ne0$ then $v\cdot x>0$ since $v$ and $x$ are non-negative vectors, and so if also $v\cdot(y'-y)>0$, then $v\cdot(x+y'-y)>v\cdot x>0$. But then $\supp(v)\cap\supp(x+y'-y)\ne\emptyset$, so Claim 2 indeed applies to $x+y'-y$.

So suppose that $i\in[n]$, and $x$ is such that $\supp(v)\cap\supp(x)\ne\emptyset$. Then by Claim $2$ there is some environment $i_1$ and some reaction $y'\to y\in\mathcal R_{i_1}$ with $v\cdot(y'-y)>0$ such that $(x^1,i_1)$ is reachable from $(x,i_1)$, where $x^1:=x+y'-y$. But $(x,i_1)$ is reachable from $(x,i)$ because $Q$ is irreducible, so $(x^1,i_1)$ is reachable from $(x,i)$. Moreover, this increase the dot product with $v$ by $v\cdot(y'-y)$; that is, $v\cdot x^1-v\cdot x=v\cdot(y'-y)$.

Iterating Claim 2, we get a sequence of environments $i_1,i_2,\cdots$ and reactions $x^1,x^2\cdots$ such that for each integer $p$, the state $(x^p,i_p)$ is reachable from $(x,i)$. Moreover, there are only finitely many reactions that can take place, so there is some minimum positive value the the dot product with $v$ can increase at each stage. Therefore, $v\cdot x^p\to\infty$ as $p\to\infty$; let $p$ be large enough that $v\cdot x^p>C$. But $(x^p,i)$ is reachable from $(x^p,i_p)$ because $Q$ is irreducible, so taking $x^0=x^p$ the claim is proven.
\end{proof}
\end{lemma}

We now assemble Lemmas \ref{lem:generator positive} and \ref{lem:v cdot x^0>C} to prove the theorem.

\begin{proof}[Proof of Theorem \ref{thm:fast-switching-transience}]
For $m\in[d]$ let $z^m$ be as in the statement of Lemma \ref{lem:generator positive}; similarly, let $I:=\supp(v)$, $h_\kappa$, and $\mathcal L_\kappa$ be as in the statement of Lemma \ref{lem:generator positive}. As in the proof of both the previous two lemmas, let $\varphi^i=(\varphi^i_1,\cdots,\varphi^i_d)$ be the vector defined by $\varphi^i_m:=v_m +\kappa^{-1} z^m_i$. Recall that $z^m\in\RR^n_{\ge0}$ was defined so that $v_m=0$ iff $z^m_i=0$. It follows that $\varphi^i_m\ge0$ for each $i,m$ and so $h_\kappa$ is a bounded function. Separately, it also follows that there exists a constant $c>1$ such that for all $m$ we have $v_m\le \varphi^i_m\le c v_m$. Let $c$ be such; notice in particular that
\begin{align*}
    v\cdot x\le \varphi^i\cdot x\le c (v\cdot x)
\end{align*}
for any $x\in\ZZ^d_{\ge0}$.

Let $x$ be such that $\supp(v)\cap\supp(x)\ne\emptyset$, and fix $i$. By Lemma \ref{lem:v cdot x^0>C}, there exists an $x^0$ such that $(x^0,i)$ is reachable from $(x,i)$ and $v\cdot x^0>c\max\{b,v\cdot x\}$ where $c$ is the constant just discussed and $b$ is the constant given by Lemma \ref{lem:generator positive}. Let $B=\{(z,j)\in\ZZ^d_{\ge0}\times[n]: v\cdot z\le \max\{b,v\cdot x\}\}$. Then by Lemma \ref{lem:generator positive}, we have $\mathcal L_\kappa h_\kappa(z,j)>0$ for every $(z,j)\notin B$.

Suppose that $(z,j)\in B$. Then
\begin{align*}
    \varphi^j\cdot z&\le c(v\cdot z)\le c\max\{b,v\cdot x\}.
\end{align*}
But by choice of $x^0$ we have
\[
c\max\{b,v\cdot x\}<v\cdot x^0\le \varphi^i\cdot x^0,
\]
so it follows that
\begin{align*}
    \varphi^j\cdot z
    &\le c\max\{b,v\cdot x\}
    <\varphi^i\cdot x^0\\
    h_\kappa(z,j)
    &\le 1-\frac1{1+c\max\{b,v\cdot x\}}
    <h_\kappa(x^0,i)
\end{align*}
Therefore, $\sup_{(z,j)\in B} h_\kappa(z,j)<h_\kappa(x^0,i)$. So by Theorem \ref{thm:Lyapunov-transience}, with positive probability the process never returns to $B$ from state $(x^0,i)$. But by construction $(x,i)\in B$, so we see that $(x^0,i)$ is a state reachable from $(x,i)$ from which with positive probability the process never returns to state $(x,i)$. It follows that $(x,i)$ is transient.

To see that $(x,i)$ is actually evanescent, notice that $B$ contains all states $(z,j)$ with $\supp(v)\cap\supp(z)=\emptyset$. In particular, every state not in $B$ is transient by above. Therefore, on the event that the process reaches $(x^0,i)$ and then never returns to $B$, the process never enters a recurrent state. By above, this event has positive probability when the process is started from state $(x,i)$, so the state $(x,i)$ is evanescent by Proposition \ref{prop:evanescent}.
\end{proof}

\subsection{Slow Switching}

Our fourth and final main theorem gives conditions under which the individual CRNs in each environment are unstable, and the mixed process inherits this behavior for slow switching rates.

\begin{theorem}\label{thm:slow-switching-transience}
In the \nameref{setting} Setting, suppose additionally that each CRN is endowed with mass-action kinetics and has at-most-monomolecular reactions. Suppose for each $i$ there exists a vector $v^i\in\RR^d_{\ge0}$ such that $(v^iM_i)_m>0$ for each $m$ such that $v^i_m>0$, where $(v^iM_i)_m$ is the $m$-th entry of the vector $v^iM_i$. Suppose further that $\supp(v^i)=\supp(v^j)$ for each $i,j\in[n]$. Then as long as $\kappa>0$ is sufficiently small, $(x,i)$ is an evanescent state for the mixed Markov chain $Z_\kappa$ whenever $\supp(v^i)\cap\supp(x)\ne\emptyset$.
\end{theorem}

\begin{remark}
    Theorem \ref{thm:slow-switching-transience}, like Theorem \ref{thm:fast-switching-transience}, differs from the recurrence theorems in that we assume that the reactions are linear instead of merely assuming that the generator is linear. Example \ref{ex:trans thms don't apply with superlinear reactions}, viewed as a CRN in one environment, shows that this assumption cannot be completely dropped from either theorem. \hfill //
\end{remark}

\begin{remark}
    The assumption in Theorem \ref{thm:slow-switching-transience} that all the vectors $v^i$ have the same support is worth of further attention. From Example \ref{ex:disjoint species}, we see that the assumption cannot be dropped entirely. However, we cannot rule out the possibility that it is enough to instead assume that $\bigcap_i\supp(v^i)\ne\emptyset$. See Conjecture \ref{conj: general conjecture} and the surrounding discussion.\hfill //
\end{remark}

Before proceeding with the proof of Theorem \ref{thm:slow-switching-transience}, we need one additional result. Just like Corollary \ref{cor:one-environment-ergodicity} was the specialization of Theorems \ref{thm:slow-switching-ergodicity} and \ref{thm:fast-switching-ergodicity} to the case with one environment (and just like Corollary \ref{cor:one-environment-transience} will be the specialization of Theorems \ref{thm:fast-switching-transience} and \ref{thm:slow-switching-transience} to the case with one environment), the following is best viewed as the one-environment case of Lemma \ref{lem:v cdot x^0>C}.

\begin{cor}\label{cor:one-environment-lemma}
Suppose that $M$ is a $d\times d$ matrix which is associated to a mass-action CRN in the sense of equation \eqref{eq:associated CRN}; suppose further that the CRN has at-most-monomolecular reactions. Suppose that there exists a vector $v\in\RR_{\ge0}^d$ such that $(vM)_m>0$ for each $m$ such that $v_m>0$, where as above $(vM)_m$ is the $m$-th entry of the vector $vM$. Then for each constant $C\in\RR$ and each $x\in\ZZ^d_{\ge0}$ such that $\supp(v)\cap\supp(x)\ne\emptyset$, there exists an $x^0\in\ZZ^d_{\ge0}$ such that $x^0$ is reachable from $x$ and $v\cdot x^0>C$.
\begin{proof}
We can view this as a switched process in one environment with $M_1=M$ and switching matrix $Q=(0)$. Applying Lemma \ref{lem:v cdot x^0>C}, the result is immediate.
\end{proof}
\end{cor}

\begin{proof}[Proof of Theorem \ref{thm:slow-switching-transience}]
For each $i\in[n]$, let $v^i$ be the vector given by the statement of the theorem. For each $x\in\ZZ^d_{\ge0}$ and each $i$, define
\[
    h(x,i)=1-\frac1{1+v^i\cdot x}.
\]
Let $\mathcal L_\kappa$ denote the generator of $Z_\kappa$.

The proof will be divided into two parts. In the first part, we will show that for each small enough $\kappa$, there exists some constant $b$ (depending implicitly on $\kappa$, but not on $x$ or $i$) such that $\mathcal L_\kappa h(x,i)>0$ whenever $v^i\cdot x\ge b$. In the second part, we will show how the result in the first part implies evanescence for the claimed states.

\textbf{Part 1:} Let $I=\supp(v^i)$; note that, by assumption, despite superficial appearances $I$ does not depend on $i$. For $i=1,\cdots,n$, let $\widetilde{\mathcal R}_i$ denote the subset of $\mathcal R_i$ consisting of reactions which do not take the $m$-th species as an input for any $m\notin I$:
\[
\widetilde{\mathcal R}_i:=\{y\to y'\in\mathcal R_i:\supp(y)\subseteq I\}
\]

We claim that if $y\to y'\notin\widetilde{\mathcal R}_i$, then $h(x-y+y',i)\ge h(x,i)$. Indeed, suppose that $y\to y'\notin\widetilde{\mathcal R}_i$. If $m\notin I$ then $v^i_m=0$ and if $m\in I$ we have $y'_m-y_m=y'_m\ge0$ (it is important here that $y$ is at-most-monomolecular). It follows that each term in the dot product in the denominator of $h(x+y'-y,i)$ is no smaller than the corresponding term in $h(x,i)$, so $h(x-y+y',i)\ge h(x,i)$ as claimed.

Notice that for each $x$ and each $i$,
\begin{align*}
    \mathcal L_\kappa h(x,i)
    &=\sum_{j\ne i}\kappa q_{ij}(h(x,j)-h(x,i))
    +\sum_{y\to y'\in\mathcal R_i}\lambda_{i,y\to y'}(x)(h(x-y+y',i)-h(x,i))\\
    &\ge\sum_{j\ne i}\kappa q_{ij}(h(x,j)-h(x,i))
    +\sum_{y\to y'\in\widetilde{\mathcal R}_i}\lambda_{i,y\to y'}(x)(h(x-y+y',i)-h(x,i))\\
    &=\frac1{1+v^i\cdot x}\left(\sum_{j\ne i}\frac{\kappa q_{ij}(v^j-v^i)\cdot x}{1+v^j\cdot x}
    +\sum_{y\to y'\in\widetilde{\mathcal R}_i}\frac{\lambda_{i,y\to y'}(x)v^i\cdot (y'-y)}{1+v^i\cdot (x-y+y')}\right)\\
    &=:g_\kappa(x,i),
\end{align*}
where the inequality follows from the discussion in the previous paragraph and the last line is the definition of $g_\kappa$.

Consider the quantity
\begin{align*}
    *:=\frac{1+v^i\cdot x}{1+v^j\cdot x}\left(\sum_{j\ne i}\kappa q_{ij}(v^j-v^i)\cdot x\right)
    +\sum_{y\to y'\in\widetilde{\mathcal R}_i}\lambda_{i,y\to y'}(x)v^i\cdot (y'-y)
\end{align*}
The difference between $*$ and $(1+v^i\cdot x)^2g_\kappa (x,i)$ is
\begin{align*}
    *-(1+v^i\cdot x)^2g_\kappa(x,i)
    &=\sum_{y\to y'\in\widetilde{\mathcal R}_i}\lambda_{i,y\to y'}(x)v^i\cdot (y'-y)\left(1-\frac{1+v^i\cdot x}{1+v^i\cdot (x-y+y')}\right)\\
    &=\sum_{y\to y'\in\widetilde{\mathcal R}_i}\lambda_{i,y\to y'}(x)v^i\cdot (y'-y)\left(\frac{v^i\cdot (y'-y)}{1+v^i\cdot (x-y+y')}\right).
\end{align*}
The term $\lambda_{i,y\to y'}(x)v^i\cdot (y'-y)$ only depends on $x$ via $\lambda_{i,y\to y'}(x)$, and since the mass-action CRN associated to $M_i$ has at-most-monomolecular reactions we know that $\lambda_{i,y\to y'}(x)$ is at most linear in $x$. Moreover, by definition of $\widetilde{\mathcal R}_i$ we know that in the linear function $\lambda_{i,y\to y'}(x)$ the coefficient of $x_m$ is zero for each $m\notin I$.

Meanwhile, the numerator of the fraction in the last line does not depend on $x$, and the denominator is (affine) linear in $x$ and each $x_m$ has a non-negative coefficient which is positive iff $m\in I$. It follows that the difference between $*$ and $(1+v^i\cdot x)^2g_\kappa(x,i)$ is bounded as $x$ varies. We conclude that if we are able to show that $*\to\infty$ as $v^i\cdot x\to\infty$, it will follow that $(1+v^i\cdot x)^2\mathcal L_\kappa h(x,i)\to\infty$ as well.

Now we turn to analyzing $*$. Notice that
\begin{align*}
    \min\left\{1,\min_{l\in I}\left\{(v^j_l)^{-1}v^i_l\right\}\right\}\left(1+\sum_{m=1}^d v^j_m x_m\right)
    &\le 1+\sum_{m=1}^d v^i_m x_m\\
    \min\left\{1,\min_{l\in I}\left\{(v^j_l)^{-1}v^i_l\right\}\right\}\left(1+v^j\cdot x\right)
    &\le 1+v^i\cdot x\\
    \min\left\{1,\min_{l\in I}\left\{(v^j_l)^{-1}v^i_l\right\}\right\}
    &\le \frac{1+v^i\cdot x}{1+v^j\cdot x}
\end{align*}
and similarly
\begin{align*}
    \frac{1+v^i\cdot x}{1+v^j\cdot x}
    &\le\max\left\{1,\max_{l\in I}\left\{(v^j_l)^{-1}v^i_l\right\}\right\}
\end{align*}
Therefore, if we define
\begin{align*}
    a^m_{i,j}=\begin{cases}\displaystyle
        \min\left\{1,\min_{l\in I}\left\{(v^j_l)^{-1}v^i_l\right\}\right\} & \qquad v^j_m-v^i_m>0\\
        1   &   \qquad v^j_m-v^i_m=0\\\displaystyle
        \max\left\{1,\max_{l\in I}\left\{(v^j_l)^{-1}v^i_l\right\}\right\} & \qquad v^j_m-v^i_m<0
    \end{cases},
\end{align*}
then
\begin{align*}
    *&=\frac{1+v^i\cdot x}{1+v^j\cdot x}\left(\sum_{j\ne i}\kappa q_{ij}\sum_{m=1}^d(v^j_m-v^i_m)x_m\right)
    +\sum_{m=1}^d v^i_m \sum_{y\to y'\in\widetilde{\mathcal R}_i}\lambda_{i,y\to y'}(x)(y'_m-y_m)\\
    &\ge \sum_{j\ne i}\kappa q_{ij}\sum_{m=1}^d a^m_{i,j}(v^j_m-v^i_m)x_m
    +\sum_{m=1}^dv^i_m\sum_{y\to y'\in\widetilde{\mathcal R}_i}\lambda_{i,y\to y'}(x)(y'_m-y_m).
\end{align*}
If the sum over $\widetilde{\mathcal R}_i$ in the last line above were over $\mathcal R_i$ instead, then it would be
\begin{align*}
    \sum_{m=1}^dv^i_m\sum_{y\to y'\in\mathcal R_i}\lambda_{i,y\to y'}(x)(y'_m-y_m)
    &=\sum_{m=1}^dv^i_m(M_ix)_m+o(\norm{x}_{\ell^1})\\
    &=v^i\cdot (M_ix)+o(\norm{x}_{\ell^1})\\
    &=(v^iM_i)\cdot x+o(\norm{x}_{\ell^1})\\
    &=\sum_{m=1}^d(v^iM_i)_m x_m+o(\norm{x}_{\ell^1}).
\end{align*}
But the distinction between $\widetilde{\mathcal R}_i$ and $\mathcal R_i$ is exactly that the former leaves out the reactions whose rate is a multiple of $x_m$ for some $m\notin I$. Therefore,
\begin{align*}
    \sum_{m=1}^dv^i_m\sum_{y\to y'\in\widetilde{\mathcal R}_i}\lambda_{i,y\to y'}(x)(y'_m-y_m)
    &=\sum_{m\in I}(v^iM_i)_m x_m+o(\norm{x}_{\ell^1}).
\end{align*}
Plugging this into the above and applying the fact that $v^j_m=0=v^i_m$ for each $i,j$ and each $m\notin I$, we get
\begin{align*}
    *&\ge \sum_{j\ne i}\kappa q_{ij}\sum_{m=1}^d a^m_{i,j}(v^j_m-v^i_m)x_m
    +\sum_{m=1}^dv^i_m\sum_{y\to y'\in\widetilde{\mathcal R}_i}\lambda_{i,y\to y'}(x)(y'_m-y_m)\\
    &=\sum_{j\ne i}\kappa q_{ij}\sum_{m=1}^d a^m_{i,j}(v^j_m-v^i_m)x_m
    +\sum_{m\in I}(v^iM_i)_m x_m+o(\norm{x}_{\ell^1})\\
    &=\sum_{j\ne i}\kappa q_{ij}\sum_{m\in I} a^m_{i,j}(v^j_m-v^i_m)x_m
    +\sum_{m\in I}(v^iM_i)_m x_m+o(\norm{x}_{\ell^1})\\
    &=\sum_{m\in I}\left(\sum_{j\ne i}\kappa q_{ij}a^m_{i,j}(v^j_m-v^i_m)+(v^iM_i)_m\right)x_m+o(\norm{x}_{\ell^1}).
\end{align*}
Therefore, to show that $*\to\infty$ as $v^i\cdot x\to\infty$, it suffices to show that
\[
    \sum_{j\ne i}\kappa q_{ij} a^m_{i,j}(v^j_m-v^i_m)+(v^iM_i)_m>0
\]
for each $m\in I$. But $(v^iM_i)_m>0$ for each $m\in I$ by assumption, and each term $\kappa q_{ij} a^m_{i,j}(v^j_m-v^i_m)$ above is converging to zero as $\kappa\to0$. Therefore, as long as $\kappa$ is sufficiently small we have that $*\to\infty$ as $v^i\cdot x\to\infty$. As discussed above, this implies that $(1+v^i\cdot x)^2\mathcal L_\kappa h(x,i)\to\infty$ as $v^i\cdot x\to\infty$. Thus as long as $v^i\cdot x$ is large enough $(1+v^i\cdot x)^2\mathcal L_\kappa h(x,i)$, and hence $\mathcal L_\kappa h(x,i)$, will be positive.

The argument above was for a particular value of $i$; applying it for each $i$ we get that as long as $\kappa$ is small enough, there is some constant $b$ (depending only on $\kappa$) such that $\mathcal L_\kappa h(x,i)>0$ whenever $v^i\cdot x>b$ (here it is important that there are only finitely many environments $i$).

\textbf{Part 2:} From Part 1, we know that there exists a number $b$ such that $\mathcal L_\kappa h(x,i)>0$ whenever $v^i\cdot x>b$. By Corollary \ref{cor:one-environment-lemma}, if $(x,i)$ is such that $\supp(v^i)\cap\supp(x)\ne\emptyset$, then there exists a $x^0\in\ZZ^d_{\ge0}$ such that $(x^0,i)$ is reachable from $(x,i)$ (indeed, reachable entirely by reactions in $\mathcal R_i$) and such that $v^i\cdot x^0>\max\{b,v^i\cdot x\}$. Let $x^0$ be such.

Now consider the set $B=\{(z,j)\in\ZZ^d_{\ge0}\times[n]:v^j\cdot z\le\max\{b,v^i\cdot x\}\}$. By choice of $b$, we have that $\mathcal L_\kappa h(z,j)>0$ for every $(z,j)\notin B$. But if $(z,j)\in B$ then
\[
v^j\cdot z\le \max\{b,v^i\cdot x\}<v^i\cdot x^0
\]
and hence
\[
h(z,j)\le1-\frac1{1+\max\{b,v^i\cdot x\}}<h(x^0,i).
\]
Therefore, $\sup_{(z,j)\in B}h(z,j)<h(x^0,i)$. So by Theorem \ref{thm:Lyapunov-transience}, with positive probability the process never returns to $B$ from state $(x^0,i)$. But by construction $(x,i)\in B$, so we see that $(x^0,i)$ is a state reachable from $(x,i)$ from which with positive probability the process never returns to state $(x,i)$. It follows that $(x,i)$ is transient, as claimed.

To see that $(x,i)$ is actually evanescent, notice that $B$ contains all states $(z,j)$ with $\supp(v^j)\cap\supp(z)=\emptyset$. In particular, every state not in $B$ is transient by above. Therefore, on the event that the process reaches $(x^0,i)$ and then never returns to $B$, the process never enters a recurrent state. By above, this event has positive probability when the process is started from state $(x,i)$, so the state $(x,i)$ is evanescent by Proposition \ref{prop:evanescent}.
\end{proof}

\subsection{One environment}

Just like the stability theorems, for instability we also have a corollary in one environment which is worth stating.

\begin{cor}\label{cor:one-environment-transience}
Suppose that $M$ is a $d\times d$ matrix which is associated to a CRN in the sense of equation \eqref{eq:associated CRN}; suppose further that the CRN is endowed with mass-action kinetics and has at-most-monomolecular reactions. Suppose that there exists a vector $v\in\RR_{\ge0}^d$ such that $(vM)_m>0$ for each $m$ such that $v_m>0$, where as above $(vM)_m$ is the $m$-th entry of the vector $vM$. Then $x$ is a evanescent state for the CRN for each $x\in\ZZ^d_{\ge0}$ such that $\supp(v)\cap\supp(x)\ne\emptyset$.
\begin{proof}
    Fix $x\in\ZZ^d_{\ge0}$ such that $\supp(v)\cap\supp(x)\ne\emptyset$. We can view the CRN as a mixed process in one environment with $M_1=M$ and switching matrix $Q=(0)$. So if we consider the mixed process with matrix $\kappa Q$, then Theorem \ref{thm:slow-switching-transience} says that $x$ is evanescent for the process if $\kappa$ is small enough and Theorem \ref{thm:fast-switching-transience} says that $x$ is evanescent for the process if $\kappa$ is large enough. But $\kappa Q=Q$ in this case, so either theorem is enough to conclude the claimed result.
\end{proof}
\end{cor}

\section{Examples}\label{sec:Examples}

This first example illustrates Theorems \ref{thm:fast-switching-ergodicity} and \ref{thm:slow-switching-transience} by applying them in a case where the switched process is transient as long as the switching is slow enough but nevertheless is exponentially ergodic when the switching is fast enough.

\begin{example}\label{ex:switching ergodic monomolecular}
Let us consider the process randomly switching, with rate $\kappa$ in each direction, between the following two mass-action SRNs
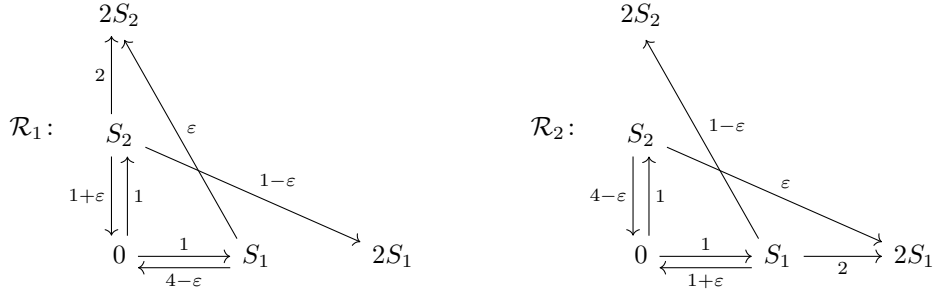
\[
\mathcal{R}_1\colon
\begin{tikzcd}[column sep=3em,row sep=3em]
2S_2&&\\
S_2 \arrow[u,shift left=0.7ex,"2"] \arrow[d,shift right=0.7ex,swap,"1+\varepsilon"]  \arrow[drr,shift left=0ex,"1-\varepsilon"]  &&\\
0 \arrow[r,"1"] \arrow[u,shift right=0.7ex,swap,"1"] & S_1 \arrow[l,shift left=1ex,"4-\varepsilon"]   \arrow[uul,shift left=0.7ex,swap,"\varepsilon"] &2S_1
\end{tikzcd}
\qquad\qquad \mathcal{R}_2\colon
\begin{tikzcd}[column sep=3em,row sep=3em]
2S_2&&\\
S_2  \arrow[d,shift right=0.7ex,swap,"4-\varepsilon"]  \arrow[drr,shift left=0ex,"\varepsilon"]  &&\\
0 \arrow[r,"1"] \arrow[u,shift right=0.7ex,swap,"1"] & S_1 \arrow[r,swap,"2"] \arrow[l,shift left=1ex,"1+\varepsilon"]   \arrow[uul,swap,shift left=0.7ex,"1-\varepsilon"] &2S_1
\end{tikzcd}
\]
\end{example}

\begin{prop}\label{prop:switching ergodic monomolecular}
In Example \ref{ex:switching ergodic monomolecular}, $\mathcal R_1$ and $\mathcal R_2$ are evanescent whenever $\varepsilon\in(0,1)$. For any $\varepsilon\in(0,1)$, the mixed process is evanescent as long as $\kappa$ is small enough, but is exponentially ergodic as long as $\kappa$ is large enough.
\begin{proof}
Notice that the respective matrices in different individual environments as well as the stationary environment are given by 
\[
    M_1=\begin{bmatrix}-4&2\varepsilon\\ 2(1-\varepsilon)&0\end{bmatrix},\quad
    M_2=\begin{bmatrix}0&2(1-\varepsilon)\\ 2\varepsilon&-4\end{bmatrix},\quad
    M=\begin{bmatrix}-2&1\\ 1&-2\end{bmatrix}
\]
Note that the eigenvalues of $M_1$ and $M_2$ are $-2\pm2\sqrt{1 + \varepsilon - \varepsilon^2}$. For $\varepsilon\in(0,1)$ we have $\varepsilon>\varepsilon^2$, and hence the larger eigenvalue is positive. Since both off-diagonal entries of $M_1$ and $M_2$ are nonzero, it follows from Proposition \ref{prop:dec/inc direction}(ii) that there exist vectors $v^1,v^2\in\RR^2_{>0}$ such that $v^iM_i\in\RR^2_{>0}$. By Corollary \ref{cor:one-environment-transience} we conclude that both $\mathcal R_1$ and $\mathcal R_2$ are evanescent for any nonzero state; the existence of inflow reactions means this state is evanescent as well. Similarly, by Theorem \ref{thm:slow-switching-transience} we conclude that all states are evanescent for the mixed process whenever $\kappa$ is sufficiently small.

However, one sees that if $v=(1,1)$ then $vM=-v$, so by Theorem \ref{thm:fast-switching-ergodicity} the mixed process converges exponentially fast whenever $\kappa$ is sufficiently large. Since the state space is irreducible, it follows by definition that the mixed process is exponentially ergodic whenever $\kappa$ is sufficiently large.
\end{proof}
\end{prop}

The next example serves to emphasize one difference between the recurrence theorems (Theorems \ref{thm:fast-switching-ergodicity} and \ref{thm:slow-switching-ergodicity}) and the transience theorems (Theorems \ref{thm:fast-switching-transience} and \ref{thm:slow-switching-transience}). Namely, this next example is an application of Theorem \ref{thm:fast-switching-ergodicity} in a case where the CRNs are not at-most-monomolecular, a possibility allowed for by the recurrence theorems but not the transience theorems.

\begin{example}\label{ex:switching ergodic non-monomolecular}
Consider the process randomly switching, with rate $\kappa$ in each direction, between the following two mass-action SRNs
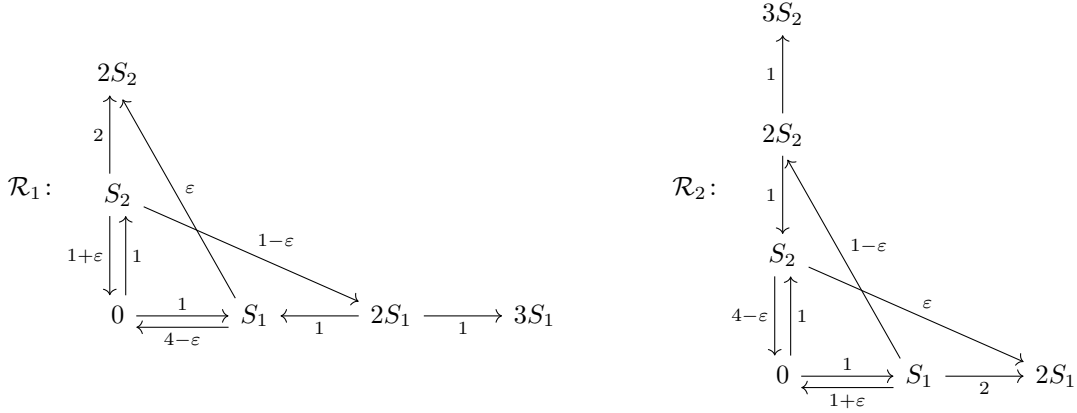
\[
\mathcal{R}_1\colon
\begin{tikzcd}[column sep=3em,row sep=3em]
2S_2&&&\\
S_2 \arrow[u,shift left=0.7ex,"2"] \arrow[d,shift right=.7ex,swap,"1+\varepsilon"]  \arrow[drr,shift left=0ex,"1-\varepsilon"]  &&&\\
0 \arrow[r,"1"] \arrow[u,shift right=0.7ex,swap,"1"] & S_1 \arrow[l,shift left=1ex,"4-\varepsilon"]   \arrow[uul,shift left=0.7ex,swap,"\varepsilon"] &2S_1\arrow[l,"1"]\arrow[r,swap,"1"] &3S_1
\end{tikzcd}
\qquad\qquad \mathcal{R}_2\colon
\begin{tikzcd}[column sep=3em,row sep=3em]
3S_2&&\\
2S_2\arrow[u,"1"]\arrow[d,swap,"1"]&&\\
S_2  \arrow[d,shift right=0.7ex,swap,"4-\varepsilon"]  \arrow[drr,shift left=0ex,"\varepsilon"]  &&\\
0 \arrow[r,"1"] \arrow[u,shift right=0.7ex,swap,"1"] & S_1 \arrow[r,swap,"2"] \arrow[l,shift left=1ex,"1+\varepsilon"]   \arrow[uul,swap,shift left=0.7ex,"1-\varepsilon"] &2S_1
\end{tikzcd}
\]
\end{example}

\begin{prop}
In Example \ref{ex:switching ergodic non-monomolecular}, for any $\varepsilon\in(0,1)$, the mixed process is exponentially ergodic as long as $\kappa$ is large enough.
\begin{proof}
The only difference between Examples \ref{ex:switching ergodic monomolecular} and \ref{ex:switching ergodic non-monomolecular} are the reactions $S_1\leftarrow 2S_1\rightarrow 3S_1$ and $S_2\leftarrow 2S_2\rightarrow 3S_2$. The terms associated to these reactions in equation \eqref{eq:associated CRN} cancel out, so Theorem \ref{thm:fast-switching-ergodicity} applies exactly as in the previous example.
\end{proof}
\end{prop}

The next example illustrates Theorems \ref{thm:fast-switching-transience} and \ref{thm:slow-switching-ergodicity} by applying them in a case where the switched process is positive recurrent as long as the switching is slow enough but the mixed process is transient as long as the switching is fast enough.

\begin{example}\label{ex:basic switching transient}
Consider the process randomly switching, with rate $\kappa$ in each direction, between the following two mass-action CRNs
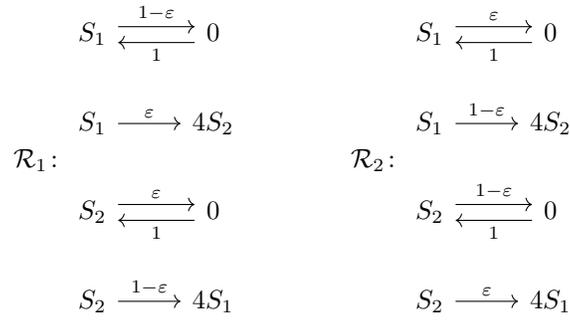
\begin{equation*}
\mathcal{R}_1\colon
\begin{tikzcd}
    S_1\arrow[yshift=3]{r}{1-\varepsilon}&0\arrow[yshift=-3]{l}{1}\\
    S_1\arrow{r}{\varepsilon}&4S_2\\
    S_2\arrow[yshift=3]{r}{\varepsilon}&0\arrow[yshift=-3]{l}{1}\\
    S_2\arrow{r}{1-\varepsilon}&4S_1
\end{tikzcd}
\qquad\qquad
\mathcal{R}_2\colon
\begin{tikzcd}
    S_1\arrow[yshift=3]{r}{\varepsilon}&0\arrow[yshift=-3]{l}{1}\\
    S_1\arrow{r}{1-\varepsilon}&4S_2\\
    S_2\arrow[yshift=3]{r}{1-\varepsilon}&0\arrow[yshift=-3]{l}{1}\\
    S_2\arrow{r}{\varepsilon}&4S_1
\end{tikzcd}
\end{equation*}
for some parameter $\varepsilon\in(0,1)$.
\end{example}

\begin{prop}
In Example \ref{ex:basic switching transient}, $\mathcal R_1$ and $\mathcal R_2$ are exponentially ergodic as long as $\varepsilon$ is small enough. For all small enough $\varepsilon$ and small enough $\kappa$ (possibly depending on $\varepsilon$) the mixed process is exponentially ergodic, but for any $\varepsilon$ the mixed process is evanescent for large enough $\kappa$ (again, possibly depending on $\varepsilon$).
\begin{proof}
The respective matrices in individual environments as well as the average environment are given by 
\[
    M_1=\begin{bmatrix}-1&4\varepsilon\\ 4(1-\varepsilon)&-1\end{bmatrix},\quad
    M_2=\begin{bmatrix}-1&4(1-\varepsilon)\\ 4\varepsilon&-1\end{bmatrix},\quad
    M=\begin{bmatrix}-1&2\\ 2&-1\end{bmatrix}
\]
Note that any $2\times 2$ matrix with negative trace and positive determinant must have two eigenvalues with negative real part. In this case $M_1$ and $M_2$ have trace $-2$ and determinant $1-16\varepsilon+16\varepsilon^2$. This determinant is continuous in $\varepsilon$ and is positive at zero, so it is positive for sufficiently small $\varepsilon$. Thus both eigenvalues of $M_1$, and both of $M_2$, have negative real part for $\varepsilon$ small. It follows from Proposition \ref{prop:dec/inc direction}(i) that for $\varepsilon$ small there exist vectors $v^1,v^2\in\RR^2_{>0}$ such that $v^iM_i\in\RR^2_{<0}$.

For such $\varepsilon$, then, applying Corollary \ref{cor:one-environment-ergodicity} tells us that $\mathcal R_1$ and $\mathcal R_2$ converge exponentially fast, and hence by irreducibility are exponentially ergodic. Similarly, applying Theorem \ref{thm:slow-switching-ergodicity} tells us that the process that switches between $\mathcal R_1$ and $\mathcal R_2$ is exponentially ergodic for any $\varepsilon$ which is small enough in the sense of the previous paragraph, provided $\kappa$ is small enough.

However, the matrix $M$ has an eigenvalue of $1$ with corresponding eigenvector $v=(1,1)$ so Theorem \ref{thm:fast-switching-transience} tells us that any state $x\in\ZZ^2_{\ge0}$ with $\supp(v)\cap\supp(x)\ne\emptyset$ is evanescent for the switching process. But $v$ is strictly positive so that includes every state except for $x=(0,0)$; the existence of at least one inflow reaction tells us that $(0,0)$ is a evanescent state as well. Thus every state is evanescent for the switching process, as claimed.
\end{proof}
\end{prop}

The next example illustrates one non-obvious way of applying Theorems \ref{thm:fast-switching-ergodicity} and \ref{thm:fast-switching-transience}.

\begin{example}\label{ex: many environments}
    Consider the following mass-action CRN:
    
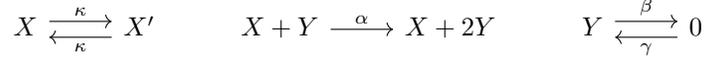
\begin{center}
    \begin{tikzcd}
        X\arrow[yshift=3]{r}{\kappa}&X'\arrow[yshift=-3]{l}{\kappa}&X+Y\arrow{r}{\alpha}&X+2Y&Y\arrow[yshift=3]{r}{\beta}&0\arrow[yshift=-3]{l}{\gamma}
    \end{tikzcd}
    \end{center}
\end{example}

\begin{prop}
Suppose the CRN from Example \ref{ex: many environments} is started in a state with $n$ total molecules of $X$ and $X'$ combined. If $n\alpha>2\beta$ and $\kappa$ is sufficiently large, then the CRN is evanescent. If $n\alpha<2\beta$ and $\kappa$ is sufficiently large, it is exponentially ergodic.
\begin{proof}
    The existence of the reaction with source complex $X+Y$ means that Theorems \ref{thm:fast-switching-ergodicity} and \ref{thm:fast-switching-transience} cannot be applied to this CRN if we view it as $3$-species network in a single environment. However, notice that the sum of the number of $X$ and $X'$ molecules is a conserved quantity. If the process is started in a state with $n$ total molecules of $X$ and $X'$ for $n\in\ZZ_{\ge0}$, then we can change perspectives and view this stochastic process as CRN with one species in $n+1$ environments, where the network has $i-1$ molecules of $X$ when in environment $i\in[n+1]=\{1,\cdots,n+1\}$. With this perspective, the environment transitions with rate matrix $\kappa Q$ for
    \begin{align*}
        Q&=
        \begin{pmatrix}
            -n & n & 0 & 0 & \cdots & 0 & 0\\
            1 & -n & n-1 & 0 & \cdots & 0 & 0\\
            0 & 2 & -n & n-2 & \cdots & 0 & 0\\
            0 & 0 & 3 &  -n  & \ddots & 0 & 0\\
            \vdots &\vdots &\vdots &\ddots &\ddots & \ddots &\vdots\\
            0 & 0 & 0 & 0 & \ddots & -n & 1\\
            0 & 0 & 0 & 0 & \cdots & n & -n\\
        \end{pmatrix},
    \end{align*}
    and for $i\in[n+1]$ we have the change of $Y$ in environment $i$ is described by the $1\times 1$ matrix
    \begin{align*}
        M_i&=((i-1)\alpha-\beta).
    \end{align*}
    Let $w_i=2^{-n}\binom{n}{i-1}$ and consider the vector $w=(w_1,\cdots,w_{n+1})$. Since $X\leftrightarrows X'$ is complex balanced, it follows from \cite[Theorem 4.1]{Anderson_Craciun_Kurtz_2010} that $w$ is the stationary distribution for $Q$; alternatively, one can check directly that
    \begin{align*}
        (wQ)_i&=w_{i-1}(n+2-i)-w_in+w_{i+1}i\\
        &=2^{-n}\left[(n+2-i)\binom n{i-2}+i\binom ni - n \binom n{i-1}\right]\\
        &=2^{-n}n\left[\binom{n-1}{i-2}+\binom{n-1}{i-1} - \binom n{i-1}\right]\\
        &=0.
    \end{align*}
    In any case, the mixed matrix $M$ is the $1\times 1$ matrix whose lone element is
    \begin{align*}
        \sum_{i=1}^{n+1} 2^{-n}\binom{n}{i-1} ((i-1)\alpha-\beta)
        &=\left(\sum_{i=0}^{n} 2^{-n}\binom{n}{i}i\right)\alpha-\beta\\
        &=\frac n2 \alpha - \beta.
    \end{align*}
    It follows from Theorem \ref{thm:fast-switching-transience} that if $n\alpha>2\beta$ then the network is evanescent for sufficiently large $\kappa$, and it follows from Theorem \ref{thm:fast-switching-ergodicity} that if $n\alpha<2\beta$ then it is exponentially ergodic for sufficiently large $\kappa$.
\end{proof}
\end{prop}

The previous examples apply Theorem \ref{thm:fast-switching-transience} in the case where the vector $v$ whose existence is assumed in the theorem statement is strictly positive. This next example is a toy example illustrating the theorem in the case where no such positive vector $v$ exists.

\begin{example}\label{ex: no strictly positive v}
    Consider the mass-action CRN consisting of $0\leftrightarrows X$ and $Y\rightarrow 2Y$, with all rate constants equal to one. This CRN is clearly transient for any state with at least one $Y$ molecule; we will show how that also follows from the theorem. In this example there is only one environment, and
    \begin{align*}
        M=M_1=\begin{pmatrix}-1&0\\0&1\end{pmatrix}.
    \end{align*}
    This matrix has eigenvalues $-1,1$. Therefore, it is Hurwitz unstable, and so by Proposition \ref{prop:Hurwitz-unstable} and Corollary \ref{cor:one-environment-transience} the CRN is transient.

    Specifically, one can check that if $v=(0,1)$ then $vM=v$ and hence by Corollary \ref{cor:one-environment-transience} the CRN is transient for any state with at least one $Y$ molecule.
    
    However, no strictly positive $v$ such that $vM$ is strictly positive can exist, since the first coordinate of $vM$ is the negation of the first coordinate of $v$.\hfill //
\end{example}

One notable difference between Theorem \ref{thm:fast-switching-ergodicity} and Theorem \ref{thm:fast-switching-transience} is that while both require that the systems are linear in the sense that they satisfy equation \eqref{eq:associated CRN}, the theorem for ergodicity allows for faster reactions whose contributions to the left-hand side of equation \eqref{eq:associated CRN} cancel out, whereas the theorem for transience does not. It is natural to wonder whether Theorem \ref{thm:fast-switching-transience} can be strengthened to allow for higher-order reactions. The next example shows that, at least in general, it cannot.

\begin{example}\label{ex:trans thms don't apply with superlinear reactions}
Consider the mass-action CRN $0\rightarrow S\rightarrow 2S\leftarrow 3S\rightarrow 4S$ with all rate constants one.
\end{example}

\begin{prop}
The CRN from Example \ref{ex:trans thms don't apply with superlinear reactions}, viewed as a CRN with one species in one environment, is associated to a matrix in the sense of equation \eqref{eq:associated CRN}. The corresponding mixed matrix is Hurwitz unstable, but nevertheless the CRN is positive recurrent.
\begin{proof}
    Notice that if $\mathcal R$ denotes the set of reactions, then
    \begin{align*}
    \sum_{y\to y'\in\mathcal R}\lambda_{y\to y'}(x)(y'_m-y_m)&=1+x-x(x-1)(x-2)+x(x-1)(x-2)\\
    &=x+1\\
    &=x+o(x).
    \end{align*}
    The corresponding matrix has a single, positive, element, namely $1$, and hence is Hurwitz unstable. So the only thing stopping us from applying Theorem \ref{thm:fast-switching-transience} or Theorem \ref{thm:slow-switching-transience} (or, for that matter, Corollary \ref{cor:one-environment-transience}) is the existence of fast reactions $2S\leftarrow 3S\rightarrow 4S$.

    However, this CRN is positive recurrent. One can see this either directly by considering the Lyapunov function $f(x)=\log(x+1)$, or by applying the criteria from \cite{Xu_Hansen_Wiuf_2022}.
\end{proof}
\end{prop}

The next example serves to illustrate the importance of one of the hypotheses of Theorems \ref{thm:slow-switching-transience}. Specifically, the theorem assumes the existence of certain vectors $v^i$ which have the same support, and the next example shows that this condition on the supports cannot be totally dropped from the theorem. In particular, it is an example where two transient CRNs are combined into a mixed process which is positive recurrent regardless of the switching rate, something which is perhaps interesting in its own right.

\begin{example}\label{ex:disjoint species}
    Let $\mathcal R_1$ and $\mathcal R_2$ be the mass-action CRNs described below:
    \begin{center}
    $\mathcal R_1$:
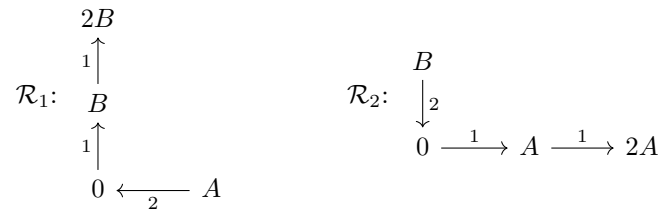
    \begin{tikzcd}
        2B\\
        B\arrow{u}{1}\\
        0\arrow{u}{1} & \arrow{l}{2}A
    \end{tikzcd}
    \qquad\qquad$\mathcal R_2$:
    \begin{tikzcd}
        B\arrow{d}{2}\\
        0\arrow{r}{1} & A \arrow{r}{1} &2A
    \end{tikzcd}
    \end{center}
    Consider the mixed process which transitions between environments $1$ and $2$ with rate $\kappa$ in each direction, and evolves according to $\mathcal R_i$ in environment $i$.
\end{example}

\begin{prop}
In Example \ref{ex:disjoint species}, there exists vectors $v^1,v^2$ such that all hypotheses of Theorem \ref{thm:slow-switching-transience} are satisfied, except that $\supp(v^1)\ne\supp(v^2)$. Nevertheless, the mixed process is positive recurrent for any choice of $\kappa$.
\begin{proof}
    If $M_i$ denotes the rate matrix for $\mathcal R_i$, then
    \begin{align*}
        M_1=\begin{pmatrix}-2&0\\0&1\end{pmatrix}
        \qquad
        M_2=\begin{pmatrix}1&0\\0&-2\end{pmatrix}
    \end{align*}
    If $v^1=(0,1)$ and $v^2=(1,0)$, then $v^iM_i=v^i$ for $i=1,2$. We see that all hypotheses of Theorem \ref{thm:slow-switching-transience} are satisfied, except that $\supp(v^1)\ne\supp(v^2)$.
    
    But fix $\kappa>0$, and define $h:\ZZ^2_{\ge0}\times\{1,2\}\to[0,\infty)$ via $h(a,b,1)=a+(1+2\kappa^{-1})b$ and $h(a,b,2)=(1+2\kappa^{-1})a+b$. Let $\mathcal L$ denote the generator of the mixed process; we see that
    \begin{align*}
        \mathcal Lh(a,b,1)&=\kappa2\kappa^{-1}(a-b)+(b+1)(1)-2a(1+2\kappa^{-1})\\
        &=-4\kappa^{-1}a-b+1.
    \end{align*}
    Similarly, $\mathcal Lh(a,b,2)=-a-4\kappa^{-1}b+1$. We conclude that $\mathcal Lh$ is negative outside a finite set, and so the model is positive recurrent, as claimed.
\end{proof}
\end{prop}

In the introduction, we summarized the heuristic of our paper as saying that when the environment switches quickly, the mixed process behaves like the average, whereas when the environment switches slowly the mixed process behaves like the processes in the individual environments. While we believe the heuristic is a good summary of the results of our four main theorems, Theorems \ref{thm:fast-switching-ergodicity}, \ref{thm:slow-switching-ergodicity}, \ref{thm:fast-switching-transience}, and \ref{thm:slow-switching-transience}, care should be taken about going too far. For instance, one might be tempted to conjecture that when some $M_i$ are unstable and others are stable, then the mixed process should be unstable for small enough $\kappa$, on the basis that the (Cartesian) product of a stable Markov chain with an unstable Markov chain is unstable. And Example \ref{ex:disjoint species} provides some support for this conjecture (think about what happens if the reaction $B\to 2B$ is added to $\mathcal R_1$ with rate constant larger than one).

However, the conjecture discussed in the previous paragraph is false. Indeed, the following example consists of two one-species at-most-molecular mass-action CRNs, one with a stable matrix and the other with an unstable matrix, such that the mixed process is stable for every choice of $\kappa>0$.

\begin{example}\label{ex:different-stability-in-each-environment}
    Let $\mathcal R_1$ and $\mathcal R_2$ be the mass-action CRNs described below:
    \begin{center}
    $\mathcal R_1$:
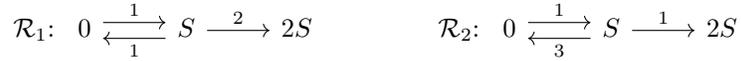
    \begin{tikzcd}
        0\arrow[yshift = 3]{r}{1} & \arrow[yshift = -3]{l}{1} S \arrow{r}{2} &   2S
    \end{tikzcd}
    \qquad\qquad$\mathcal R_2$:
    \begin{tikzcd}
        0\arrow[yshift = 3]{r}{1} & \arrow[yshift = -3]{l}{3} S \arrow{r}{1} &   2S
    \end{tikzcd}
    \end{center}
    Consider the mixed process which transitions between environments $1$ and $2$ with rate $\kappa$ in each direction, and evolves according to $\mathcal R_i$ in environment $i$.
\end{example}

\begin{prop}
In Example \ref{ex:different-stability-in-each-environment}, the mixed process is exponentially ergodic for any choice of $\kappa>0$.
\begin{proof}
    Fix $\kappa>0$, and consider the function defined by
    \[
        h_\kappa(x,i)=\begin{cases}
            90 x^{\kappa/3} & i=1\\
            57 x^{\kappa/3} & i=2
        \end{cases}.
    \]
    Notice that
    \begin{align*}
        \mathcal L_\kappa h_\kappa(x,1)
        &=(2x+1)(90(x+1)^{\kappa/3}-90x^{\kappa/3})
        +x(90(x-1)^{\kappa/3}-90x^{\kappa/3})
        +\kappa(57x^{\kappa/3}-90x^{\kappa/3})\\
        &=(2x+1)(30\kappa x^{\kappa/3-1}+o(x^{\kappa/3-1}))
        -x(30\kappa x^{\kappa/3-1}+o(x^{\kappa/3-1}))
        -33\kappa x^{\kappa/3}\\
        &=-3\kappa x^{\kappa/3}+o(x^{\kappa/3}),
    \end{align*}
    and similarly
    \begin{align*}
        \mathcal L_\kappa h_\kappa(x,2)
        &=(x+1)(57(x+1)^{\kappa/3}-57x^{\kappa/3})
        +3x(57(x-1)^{\kappa/3}-57x^{\kappa/3})
        +\kappa(90x^{\kappa/3}-57x^{\kappa/3})\\
        &=(x+1)(19\kappa x^{\kappa/3-1}+o(x^{\kappa/3-1}))
        -3x(19\kappa x^{\kappa/3-1}+o(x^{\kappa/3-1}))
        +33\kappa x^{\kappa/3}\\
        &=-5\kappa x^{\kappa/3}+o(x^{\kappa/3}).
    \end{align*}
    The desired result follows from Theorem \ref{thm:Lyapunov-exponential-ergodicity}.
\end{proof}
\end{prop}

\section{Intermediate Switching}\label{sec:Intermediate Switching}

The purpose of this section is to explore possible dynamics for intermediate values of $\kappa$. First, we will show that there are models in which the behavior at sufficiently small and sufficiently large values of $\kappa$ is the same, but at some intermediate $\kappa$ the behavior changes. Specifically, the model in Example \ref{ex:trans-large-and-small} will turn out to be transient for both small and large $\kappa$ but exponentially ergodic for intermediate $\kappa$, whereas the model in Example \ref{ex:ergodic-large-and-small} will be exponentially ergodic for both small and large $\kappa$ but transient for intermediate $\kappa$. Second, building on these examples, in Proposition \ref{prop:arbitrarily-many-phase-transitions} we will show that there exists models with an arbitrarily large number of phase transitions between exponentially ergodic and transient as the parameter $\kappa$ is varied, with our preferred behavior at extreme values of $\kappa$.

Recall that Examples \ref{ex:switching ergodic monomolecular} and \ref{ex:basic switching transient} show that there exists models which are transient for small $\kappa$ and exponentially ergodic for large $\kappa$, and vice versa. If these networks are combined into one big network with two non-interacting sets of species, one from the first example and one from the second, the combined network will be transient whenever either model is transient, and exponentially ergodic when both are exponentially ergodic. Thus to construct an example which is transient for $\kappa$ both small and large but exponentially ergodic for intermediate $\kappa$, we just have to combine these two previous examples in such a way that regimes where both are exponentially ergodic overlap. As it turns out, this is straightforward to do by simply adjusting the relative rates of the two models, which is the content of the following example.

\begin{example}\label{ex:trans-large-and-small}
Let us consider the process randomly switching, with rate $\kappa$ in each direction, between the following two mass-action SRNs
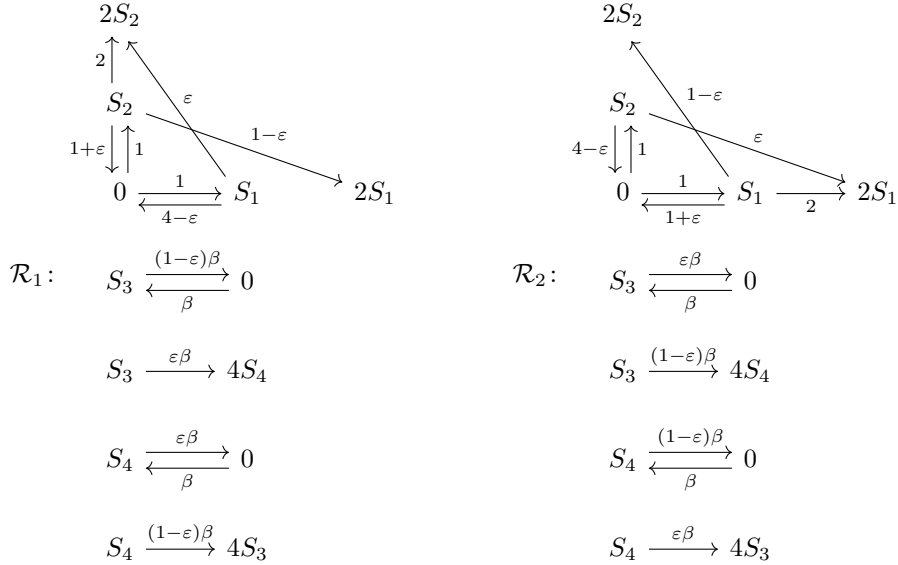
\begin{equation*}
\mathcal{R}_1\colon
\begin{tikzcd}
    2S_2&&\\
    S_2 \arrow[shift left=0.7ex]{u}{2} \arrow[shift right=0.7ex,swap]{d}{1+\varepsilon}  \arrow[shift left=0ex]{drr}{1-\varepsilon}  &&\\
    0 \arrow{r}{1} \arrow[shift right=0.7ex,swap]{u}{1} & S_1 \arrow[shift left=1ex]{l}{4-\varepsilon}   \arrow[shift left=0.7ex,swap]{uul}{\varepsilon} &2S_1\\
    S_3\arrow[yshift=3]{r}{(1-\varepsilon)\beta}&0\arrow[yshift=-3]{l}{\beta}&\\
    S_3\arrow{r}{\varepsilon\beta}&4S_4&\\
    S_4\arrow[yshift=3]{r}{\varepsilon\beta}&0\arrow[yshift=-3]{l}{\beta}&\\
    S_4\arrow{r}{(1-\varepsilon)\beta}&4S_3&
\end{tikzcd}
\qquad\qquad
\mathcal{R}_2\colon
\begin{tikzcd}
    2S_2&&\\
    S_2  \arrow[shift right=0.7ex,swap]{d}{4-\varepsilon}  \arrow[shift left=0ex]{drr}{\varepsilon}  &&\\
    0 \arrow{r}{1} \arrow[shift right=0.7ex,swap]{u}{1} & S_1 \arrow[swap]{r}{2} \arrow[shift left=1ex]{l}{1+\varepsilon}   \arrow[swap,shift left=0.7ex]{uul}{1-\varepsilon} &2S_1\\
    S_3\arrow[yshift=3]{r}{\varepsilon\beta}&0\arrow[yshift=-3]{l}{\beta}&\\
    S_3\arrow{r}{(1-\varepsilon)\beta}&4S_4&\\
    S_4\arrow[yshift=3]{r}{(1-\varepsilon)\beta}&0\arrow[yshift=-3]{l}{\beta}&\\
    S_4\arrow{r}{\varepsilon\beta}&4S_3&
\end{tikzcd}
\end{equation*}
for some parameters $\varepsilon\in(0,1)$ and $\beta>0$.
\end{example}

\begin{prop}\label{prop:trans-large-and-small}
    There exists a choice of $\varepsilon$ and $\beta$ such that every state is evanescent for the model in Example \ref{ex:trans-large-and-small} for both $\kappa$ sufficiently large and $\kappa$ sufficiently small, but such that there nevertheless exist intermediate values of $\kappa$ for which the model is exponentially ergodic.
\begin{proof}
Notice that this model consists of two two-species submodels which do not interact. By Proposition \ref{prop:switching ergodic monomolecular}, for any $\varepsilon\in(0,1)$ the first submodel (the one with species $S_1,S_2$) is evanescent when $\kappa$ is small enough and exponentially ergodic when $\kappa$ is large enough. Let $\kappa_{\min}$ be large enough (implicitly depending on $\varepsilon$) that the first submodel is exponentially ergodic whenever $\kappa>\kappa_{\min}$.

Similarly, the second submodel (the one with species $S_3,S_4$) is a version of the model from Example \ref{ex:basic switching transient}, but with the speeds scaled up ($\beta>1$) or down ($\beta<1$) by a factor of $\beta$. Fix $\varepsilon\in(0,1)$ going forward to be small enough that the model from Example \ref{ex:basic switching transient} is exponentially ergodic for small $\kappa$, and let $\kappa_{\max}$ be small enough that said model is exponentially ergodic whenever $\kappa<\kappa_{\max}$. It follows that the second submodel here is exponentially ergodic whenever $\kappa\beta^{-1}<\kappa_{\max}$ but is transient whenever $\kappa\beta^{-1}$ is sufficiently large (see Figure \ref{fig:rescaling} for more explanation).

\begin{figure}
    \centering
    $\mathcal R^1_1\colon$
    \begin{tikzcd}
        S_3\arrow[yshift=3]{r}{1-\varepsilon}&0\arrow[yshift=-3]{l}{1}&\\
        S_3\arrow{r}{\varepsilon}&4S_4&\\
        S_4\arrow[yshift=3]{r}{\varepsilon}&0\arrow[yshift=-3]{l}{1}&\\
        S_4\arrow{r}{1-\varepsilon}&4S_3&
    \end{tikzcd}
    \qquad\qquad
    $\mathcal R^1_2\colon$
    \begin{tikzcd}
        S_3\arrow[yshift=3]{r}{\varepsilon}&0\arrow[yshift=-3]{l}{1}&\\
        S_3\arrow{r}{1-\varepsilon}&4S_4&\\
        S_4\arrow[yshift=3]{r}{1-\varepsilon}&0\arrow[yshift=-3]{l}{1}&\\
        S_4\arrow{r}{\varepsilon}&4S_3&
    \end{tikzcd}
    \qquad\qquad
    $Q_1=\begin{pmatrix}-1&1\\1&-1\end{pmatrix}$
    
    \vspace{1cm}
    
    $\mathcal R^2_1\colon$
    \begin{tikzcd}
        S_3\arrow[yshift=3]{r}{(1-\varepsilon)\beta}&0\arrow[yshift=-3]{l}{\beta}&\\
        S_3\arrow{r}{\varepsilon\beta}&4S_4&\\
        S_4\arrow[yshift=3]{r}{\varepsilon\beta}&0\arrow[yshift=-3]{l}{\beta}&\\
        S_4\arrow{r}{(1-\varepsilon)\beta}&4S_3&
    \end{tikzcd}
    \qquad\qquad
    $\mathcal R^2_2\colon$
    \begin{tikzcd}
        S_3\arrow[yshift=3]{r}{\varepsilon\beta}&0\arrow[yshift=-3]{l}{\beta}&\\
        S_3\arrow{r}{(1-\varepsilon)\beta}&4S_4&\\
        S_4\arrow[yshift=3]{r}{(1-\varepsilon)\beta}&0\arrow[yshift=-3]{l}{\beta}&\\
        S_4\arrow{r}{\varepsilon\beta}&4S_3&
    \end{tikzcd}
    \qquad\qquad
    $Q_2=\begin{pmatrix}-\beta&\beta\\\beta&-\beta\end{pmatrix}$
    \caption{Two models, each with two environments and two species. The first model consists of the networks $\mathcal R^1_1$ and $\mathcal R^1_2$ and environment rate matrix $(\kappa\beta^{-1})Q_1$, and the second model consists of $\mathcal R^2_1$ and $\mathcal R^2_2$ and rate matrix $(\kappa\beta^{-1})Q_2$. The first model is exactly the one from Example \ref{ex:basic switching transient} but with the species reindexed and with $\kappa\beta^{-1}$ as a parameter in place of $\kappa$. Meanwhile, the second model is exactly the restriction to species $\{S_3,S_4\}$ of the model considered in Example \ref{ex:trans-large-and-small} (even the rate matrix $(\kappa\beta^{-1})Q_2$ is the same, since you can distribute the factor of $\beta^{-1}$ to cancel the $\beta$s inside $Q_2$).
    But notice that the second model is exactly the first model but will all transition rates multiplied by a factor of $\beta$. Therefore, the two models will have the same long-term stability behavior. In particular, the first model is exponentially ergodic (respectively, evanescent) for a given value of $(\kappa\beta^{-1})$ iff the second is.}
    \label{fig:rescaling}
\end{figure}
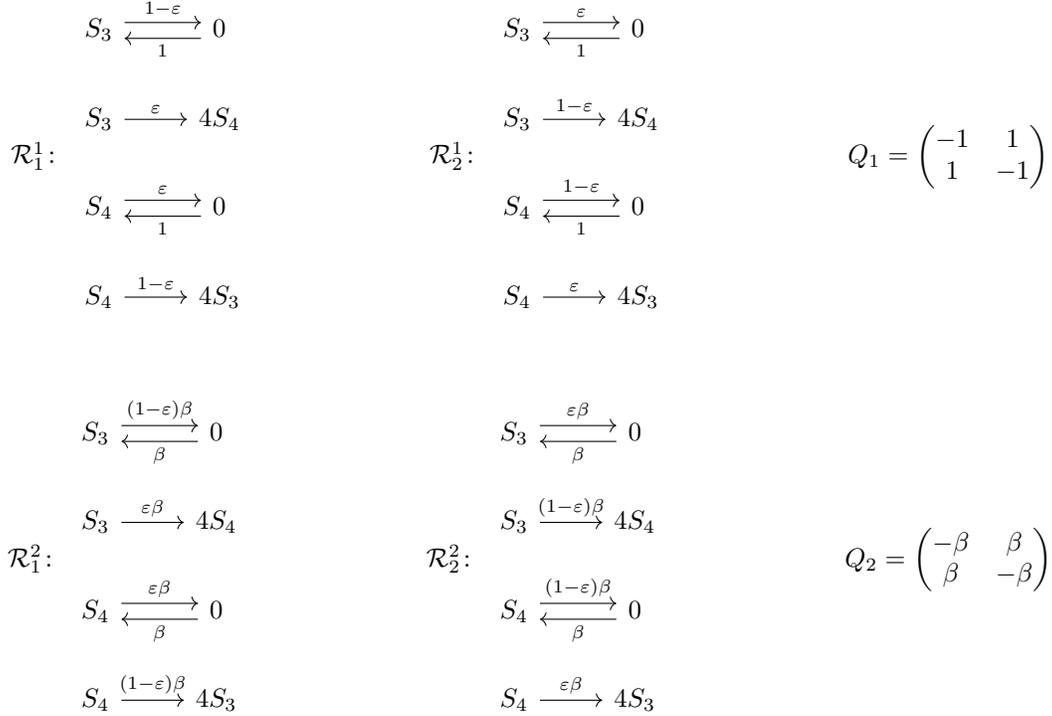

Now fix $\beta$ to be large enough that $\kappa_{\min}<\beta\kappa_{\max}$. Notice that when $\kappa$ is sufficiently small, the first submodel is evanescent and hence so is the entire model. Similarly, when $\kappa$ is sufficiently large the second submodel is evanescent and hence so is the entire model. But for any $\kappa\in(\kappa_{\min},\beta\kappa_{\max})$, both submodels are exponentially ergodic and hence the entire model is exponentially ergodic.
\end{proof}
\end{prop}

Note a fundamental asymmetry in the proof of Proposition \ref{prop:trans-large-and-small}: The combined ($4$-species) model is transient when \emph{either} of the two ($2$-species) submodels is transient, but is exponentially ergodic only when \emph{both} of the submodels are exponentially ergodic. As a consequence, it will not be possible to use the same trick to produce a model which is exponentially ergodic for all sufficiently large $\kappa$ and also all sufficiently small $\kappa$, but is nevertheless transient for intermediate $\kappa$. We will create such a model via other methods, detailed below. In any case, once we do have such a model, the same trick used in Proposition \ref{prop:trans-large-and-small} will allow us to create models with an arbitrarily large number of phase transitions and our choice of behavior (exponentially ergodic or transient) for sufficiently large $\kappa$ and separately for sufficiently small $\kappa$ (see Proposition \ref{prop:arbitrarily-many-phase-transitions}).

To construct examples of models which are exponentially ergodic for both sufficiently large $\kappa$ and sufficiently small $\kappa$, but are transient for intermediate $\kappa$, we will make use of the following proposition. The proposition is intentionally stated only in the minimum generality necessary to verify Example \ref{ex:ergodic-large-and-small} while abstracting away unnecessary detail.

\begin{prop}\label{prop:ergodic-large-and-small}
Consider a \nameref{setting} setting model with $d$ species and $4$ environments. Assume mass-action kinetics and at-most-monomolecular reactions. Suppose that the switching between environments happens according to $\kappa Q$ for
\[
Q=(q_{ij})_{i,j=1}^4
:=\begin{pmatrix}
    -(1+2\varepsilon) & 1 & \varepsilon & \varepsilon\\
    1 & -1(1+2\varepsilon) & \varepsilon & \varepsilon\\
    \varepsilon & \varepsilon & -(1+2\varepsilon) & 1\\
    \varepsilon & \varepsilon & 1 & -(1+2\varepsilon)\\
\end{pmatrix}
\]
for some $\kappa,\varepsilon>0$. Suppose further that $A_1:=(M_1+M_2)/2$ and $A_2:=(M_3+M_4)/2$ and $v^1,v^2\in\RR^d_{>0}$ are such that $v^iA_i\in\RR^d_{>0}$ for $i=1,2$. Then there exists a choice of $\kappa$ and $\varepsilon$ such that every $(x,i)\in\ZZ^d_{\ge 0}\times[4]$ with $x\ne0$ is an evanescent state.
\end{prop}

\begin{remark}
    As mentioned, Proposition \ref{prop:ergodic-large-and-small} is not stated in the maximum possible generality. The proposition involves two groups of two environments, with fast switching within groups and slow switching between groups. But there is no reason other than ease of notation to assume that there are only two groups and/or only two environments within each group, nor even that each group has the same number of environments. Similarly, only convenience leads us to assume that the process divides its time equally between and within the groups. Lastly, note that proposition requires $v^1,v^2\in\RR^d_{>0}$; it should be possible to allow for $v^1,v^2\in\RR^d_{\ge0}$ in the manner of Theorems \ref{thm:fast-switching-transience} and \ref{thm:slow-switching-transience} by following the proof of either theorem.
\end{remark}

\begin{proof}[Proof of Proposition \ref{prop:ergodic-large-and-small}]
We begin by applying Lemma \ref{lem:z^m}. However, rather than applying it to the particular four-environment system that is the subject of this theorem, we instead apply it to the two-environment system with CRNs associated to $M_1$ and $M_2$, and environment Q-matrix
\[
\begin{pmatrix}-1&1\\1&-1\end{pmatrix}.
\]
Taking $v:=v^1$ in the statement of Lemma \ref{lem:z^m}, for each $m\in[d]$ we get $z^m_1,z^m_2>0$ such that
\[
z^m_2-z^m_1+(v^1M_1)_m
=(v^1A_1)_m
=z^m_1-z^m_2+(v^1M_2)_m.
\]
Doing the same for matrices $M_3,M_4$ and vector $v^2$, for each $m\in[d]$ we get $z^m_3,z^m_4>0$ such that
\[
z^m_4-z^m_3+(v^2M_3)_m
=(v^2A_2)_m
=z^m_3-z^m_4+(v^2M_4)_m.
\]
Now define $h_\kappa:\ZZ_{\ge0}^d\times[4]\to\RR$ via
\[
h_\kappa(x,i)=
1-\frac1{1+(v^{\ceil{i/2}}+\kappa^{-1}u^i)\cdot x},
\]
where $\ceil{i/2}$ is the smallest integer greater than or equal to $i/2$ and $u^i\in\RR^d_{>0}$ is the vector defined by $u^i_m:=z^m_i$. To make our notation more concise, for $i\in[4]$ we let $\varphi^i\in\RR^d_{>0}$ denote the coefficient vector $\varphi^i:=v^{\ceil{i/2}}+\kappa^{-1}u^i$ from the denominator of $h_\kappa$.

The proof will be divided into two parts. In the first part, we will show that for some choice of $\kappa$ and $\varepsilon$, there exists a constant $b$ (depending implicitly on $\kappa$ and $\varepsilon$, but not on $x$ or $i$) such that $\mathcal L_\kappa h_\kappa(x,i)>0$ whenever $v^{\ceil{i/2}}\cdot x\ge b$. In the second part, we will show how the result in the first part implies evanescence for the claimed states.

\textbf{Part 1:} Fix $x$ and $i$, and let $\widehat\imath$ denote the index of the environment which is paired with environment $i$:
\[
\widehat\imath:=\begin{cases}
    2   &   i=1\\
    1   &   i=2\\
    4   &   i=3\\
    3   &   i=4
\end{cases}.
\]
Notice that
\begin{align*}
    \mathcal L_\kappa h_\kappa(x,i)
    &=\sum_{j\ne i}\kappa q_{ij}(h_\kappa(x,j)-h_\kappa(x,i))
    +\sum_{y\to y'\in\mathcal R_i}\lambda_{i,y\to y'}(x)(h_\kappa(x-y+y',i)-h_\kappa(x,i))\\
    &=\frac{1}{1+\varphi^i\cdot x}\Bigg(\frac{\kappa(\kappa^{-1}u^{\widehat\imath}-\kappa^{-1}u^i)\cdot x}{1+\varphi^{\widehat\imath}\cdot x}+\sum_{j\ne i,\widehat\imath}\frac{\kappa \varepsilon (\varphi^j-\varphi^i)\cdot x}{1+\varphi^j\cdot x}+\sum_{y\to y'\in\mathcal R_i}\frac{\lambda_{i,y\to y'}(x)\varphi^i\cdot(y'-y)}{1+\varphi^i\cdot(x-y+y')}\Bigg)
\end{align*}
Consider the quantity
\begin{align*}
    *&:=(1+\varphi^i\cdot x)\left(\frac{\kappa(\kappa^{-1}u^{\widehat\imath}-\kappa^{-1}u^i)\cdot x}{1+\varphi^{\widehat\imath}\cdot x}+\sum_{j\ne i,\widehat\imath}\frac{\kappa \varepsilon (\varphi^j-\varphi^i)\cdot x}{1+\varphi^j\cdot x}\right)+\sum_{y\to y'\in\mathcal R_i}\lambda_{i,y\to y'}(x)\varphi^i\cdot(y'-y)
\end{align*}
The difference between $*$ and $(1+\varphi^i\cdot x)^2 \mathcal L_\kappa h_\kappa(x,i)$ is
\begin{align*}
    *\,-\,\left(1+\varphi^i\cdot x\right)^2 \mathcal L_\kappa h_\kappa(x,i)
    &=\sum_{y\to y'\in\mathcal R_i}\lambda_{i,y\to y'}(x)\varphi^i\cdot(y'-y)\left(1-\frac{1+\varphi^i\cdot x}{1+\varphi^i\cdot(x-y+y')}\right)\\
    &=\sum_{y\to y'\in\mathcal R_i}\lambda_{i,y\to y'}(x)\varphi^i\cdot(y'-y) \left(\frac{\varphi^i\cdot(y'-y)}{1+\varphi^i\cdot(x-y+y')}\right).
\end{align*}
The term $\lambda_{i,y\to y'}(x)\varphi^i\cdot(y'-y)$ only depends on $x$ via $\lambda_{i,y\to y'}(x)$, and since the mass-action CRN associated to $M_i$ has at-most-monomolecular reactions we know that $\lambda_{i,y\to y'}(x)$ is at most linear in $x$. Meanwhile, the numerator of the fraction in the last line does not depend on $x$, and the denominator is (affine) linear in $x$ with the coefficient of each $x_m$ strictly positive. Therefore, each term in the (finite!) sum above is bounded as $x_m$ varies, and so the difference between $*$ and $(1+\varphi^i\cdot x)^2 \mathcal L_\kappa h_\kappa(x,i)$ is bounded. We conclude that if we are able to show that $*\to\infty$ as $\norm x_{\ell^1}\to\infty$ (or equivalently, as $v^{\ceil{i/2}}\cdot x\to\infty$), then it would follow that $(1+\varphi^i\cdot x)^2 \mathcal L_\kappa h_\kappa(x,i)\to\infty$ as well.

Now we turn to analyzing $*$. Notice that
\begin{align*}
    \min\left\{1,\min_{l\in [d]}\left\{\frac{\varphi^i_l}{\varphi^j_l}\right\}\right\}\left(1+\varphi^j\cdot x\right)
    &=\min\left\{1,\min_{l\in [d]}\left\{\frac{\varphi^i_l}{\varphi^j_l}\right\}\right\}\left(1+\sum\limits_{m=1}^d \varphi^j_mx_m\right)
    \le 1+\sum\limits_{m=1}^d \varphi^i_mx_m
    = 1+\varphi^i\cdot x\\
    \min\left\{1,\min_{l\in [d]}\left\{\frac{\varphi^i_l}{\varphi^j_l}\right\}\right\}
    &\le \frac{1+\varphi^i\cdot x}{1+\varphi^j\cdot x}
\end{align*}
and similarly
\begin{align*}
    \frac{1+\varphi^i\cdot x}{1+\varphi^j\cdot x}
    &\le \max\left\{1,\max_{l\in [d]}\left\{\frac{\varphi^i_l}{\varphi^j_l}\right\}\right\}
\end{align*}
Therefore, if we define
\begin{align*}
    a^{i,j}_m=\begin{cases}\displaystyle
        \min\left\{1,\min_{l\in [d]}\left\{\frac{\varphi^i_l}{\varphi^j_l}\right\}\right\}      &     \qquad \varphi^j_m-\varphi^i_m>0\\
        1   &     \qquad \varphi^j_m-\varphi^i_m=0\\\displaystyle
       \max\left\{1,\max_{l\in [d]}\left\{\frac{\varphi^i_l}{\varphi^j_l}\right\}\right\}      &     \qquad \varphi^j_m-\varphi^i_m<0
    \end{cases},
\end{align*}
for $j\in[4]\setminus\{i\}$, then (since $\varphi^{\widehat\imath}-\varphi^i=\kappa^{-1}(u^{\widehat\imath}-u^i)$)
\begin{align*}
    *&=(1+\varphi^i\cdot x)\left(\frac{\kappa(\kappa^{-1}u^{\widehat\imath}-\kappa^{-1}u^i)\cdot x}{1+\varphi^{\widehat\imath}\cdot x}+\sum_{j\ne i,\widehat\imath}\frac{\kappa \varepsilon (\varphi^j-\varphi^i)\cdot x}{1+\varphi^j\cdot x}\right)+\sum_{y\to y'\in\mathcal R_i}\lambda_{i,y\to y'}(x)\varphi^i\cdot(y'-y)\\
    &\ge \sum_{m=1}^d a^{i,\widehat\imath}_m(u^{\widehat\imath}_m-u^i_m)x_m+\sum_{j\ne i,\widehat\imath}\sum_{m=1}^d\kappa\varepsilon a^{i,j}_m(\varphi^j_m-\varphi^i_m)x_m+\sum_{y\to y'\in\mathcal R_i}\lambda_{i,y\to y'}(x)\varphi^i\cdot(y'-y)\\
    &=\sum_{m=1}^d\left(a^{i,\widehat\imath}_m(u^{\widehat\imath}_m-u^i_m)+\sum_{j\ne i,\widehat\imath}\kappa\varepsilon a^{i,j}_m(\varphi^j_m-\varphi^i_m)\right)x_m+\sum_{y\to y'\in\mathcal R_i}\lambda_{i,y\to y'}(x)\varphi^i\cdot(y'-y).
\end{align*}
Notice that
\begin{align*}
    \sum_{y\to y'\in\mathcal R_i}\lambda_{i,y\to y'}(x)\varphi^i\cdot(y'-y)
    &=\sum_{y\to y'\in\mathcal R_i}\lambda_{i,y\to y'}(x)\sum_{m=1}^d \varphi^i_m(y'_m-y_m)\\
    &=\sum_{m=1}^d \varphi^i_m(M_ix)_m+o(\norm{x}_{\ell^1})\\
    &=\varphi^i\cdot(M_ix)+o(\norm{x}_{\ell^1})\\
    &=(\varphi^i M_i)\cdot x+o(\norm{x}_{\ell^1})\\
    &=\sum_{m=1}^d(\varphi^i M_i)_mx_m+o(\norm{x}_{\ell^1}).
\end{align*}
Plugging this into the above, we get
\begin{align*}
    *
    &\ge\sum_{m=1}^d\left(a^{i,\widehat\imath}_m(u^{\widehat\imath}_m-u^i_m)+\sum_{j\ne i,\widehat\imath}\kappa\varepsilon a^{i,j}_m(\varphi^j_m-\varphi^i_m)\right)x_m+\sum_{y\to y'\in\mathcal R_i}\lambda_{i,y\to y'}(x)\varphi^i\cdot(y'-y)\\
    &=\sum_{m=1}^d\left(a^{i,\widehat\imath}_m(u^{\widehat\imath}_m-u^i_m)+(\varphi^i M_i)_m+\sum_{j\ne i,\widehat\imath}\kappa\varepsilon a^{i,j}_m(\varphi^j_m-\varphi^i_m)\right)x_m+o(\norm{x}_{\ell^1}).
\end{align*}
Written in this form, it is clear that to show $*\to\infty$ as $v^{\ceil{i/2}}\cdot x\to\infty$, it suffices to show that
\begin{align}\label{eq:leading coefficient intermediate switching}
    a^{i,\widehat\imath}_m(u^{\widehat\imath}_m-u^i_m)+(\varphi^i M_i)_m+\sum_{j\ne i,\widehat\imath}\kappa\varepsilon a^{i,j}_m(\varphi^j_m-\varphi^i_m)>0
\end{align}
for each $m\in [d]$. But notice that $a^{i,\widehat\imath}_m\to1$ as $\kappa\to\infty$ and $\varphi_m^i\to v^{\ceil{i/2}}_m$ as $\kappa\to\infty$, and the expression $a^{i,\widehat\imath}_m(u^{\widehat\imath}_m-u^i_m)+(\varphi^i M_i)_m$ does not otherwise depend on $\kappa$. So
\begin{align*}
    \lim_{\kappa\to\infty}a^{i,\widehat\imath}_m(u^{\widehat\imath}_m-u^i_m)+(\varphi^i M_i)_m
    &=(z_{\widehat\imath}^m-z_i^m)+(v^{\ceil{i/2}} M_i)_m
    =(v^{\ceil{i/2}}A_{\ceil{i/2}})_m,
\end{align*}
where the second equality is a consequence of how we defined $z^m$ at the beginning of the proof, using Lemma \ref{lem:z^m}. But both vectors $v^1A_1$ and $v^2A_2$ have all entries strictly positive by assumption, and so we conclude that $a^{i,\widehat\imath}_m(u^{\widehat\imath}_m-u^i_m)+(\varphi^i M_i)_m$ is positive for sufficiently large $\kappa$. Choose $\kappa$ large enough that
\[
a^{i,\widehat\imath}_m(u^{\widehat\imath}_m-u^i_m)+(\varphi^i M_i)_m>0
\]
for all $m\in[d]$, and notice that the expression $\kappa a^{i,j}_m(\varphi^j_m-\varphi^i_m)$ depends on a few things ($\kappa$, $j$, $m$, etc.) but crucially not on $\varepsilon$. So
\[
\lim_{\varepsilon\to0}\sum_{j\ne i,\widehat\imath}\kappa\varepsilon a^{i,j}_m(\varphi^j_m-\varphi^i_m)=0;
\]
pick $\varepsilon>0$ small enough that
\begin{align*}
    \sum_{j\ne i,\widehat\imath}\kappa\varepsilon a^{i,j}_m(\varphi^j_m-\varphi^i_m)>-\left(a^{i,\widehat\imath}_m(u^{\widehat\imath}_m-u^i_m)+(\varphi^i M_i)_m\right)
\end{align*}
for every $m\in[d]$. With this choice of $\kappa,\varepsilon>0$, inequality \eqref{eq:leading coefficient intermediate switching} is satisfied. It follows that we have $*\to\infty$ as $v^{\ceil{i/2}}\cdot x\to\infty$ and hence $(1+\varphi^i\cdot x)^2 \mathcal L_\kappa h_\kappa(x,i)\to\infty$ as well. Fix $b$ large enough that when $v^{\ceil{i/2}}\cdot x\ge b$, we have $(1+\varphi^i\cdot x)^2 \mathcal L_\kappa h_\kappa(x,i)>0$.  Then because the square is positive, for all such $x$ we have $\mathcal L_\kappa h_\kappa(x,i)>0$ as well. (In principal $b$ should depend on $i$ as well as on $\kappa$ and $\varepsilon$, but there are only four possibilities for $i$ so we're just picking $b$ large enough that it works for all $i$). In other words, $\mathcal L_\kappa h_\kappa(x,i)>0$ whenever $v^{\ceil{i/2}}\cdot x\ge b$, which is what we said we would prove in Part 1.

\textbf{Part 2:} Fix $(x,i)\in\ZZ^d_{\ge 0}\times[4]$ with $x\ne0$, and let $\widehat\imath$ be as in Part 1. Note that since $\varphi^j$ and $v^{\ceil{j/2}}$ are both strictly positive vectors for all $j$, there exists some constant $c\ge 1$ such that for every $(z,j)\in\ZZ^d_{\ge 0}\times[4]$ we have $\varphi^j\cdot z\le c\left(v^{\ceil{j/2}}\cdot z\right)$. Let $c$ be such.

We begin this part in earnest by applying Lemma \ref{lem:v cdot x^0>C}. However, rather than applying it to the particular four-environment system that is the subject of this theorem, we instead apply it to the two-environment system with CRNs associated to $M_i$ and $M_{\widehat\imath}$, and environment Q-matrix
\[
\begin{pmatrix}-1&1\\1&-1\end{pmatrix}.
\]
This tells us that there exists an $x^0\in\ZZ^d_{\ge0}$ such that $(x^0,i)$ is reachable from $(x,i)$ (indeed, reachable entirely by reactions in $\mathcal R_i$ and $\mathcal R_{\widehat\imath}$) and $v^{\ceil{i/2}}\cdot x^0>c\max\{b,v^{\ceil{i/2}}\cdot x\}$, where $c$ is the constant discussed at the beginning of Part 2 and $b$ is the constant given to us by Part 1.

Now consider the set $B=\{(z,j)\in\ZZ^d_{\ge0}\times[4]:v^{\ceil{j/2}}\cdot z\le \max\{b,v^{\ceil{i/2}}\cdot x\}\}$. By choice of $b$, we have that $\mathcal L_\kappa h_\kappa(z,j)>0$ for every $(z,j)\notin B$. But if $(z,j)\in B$ then
\begin{align*}
    \varphi^j\cdot z
    &\le c\left(v^{\ceil{j/2}}\cdot z\right)
    \le c\max\{b,v^{\ceil{i/2}}\cdot x\}.
\end{align*}
But by choice of $x^0$ we have
\[
c\max\{b,v^{\ceil{i/2}}\cdot x\}
<v^{\ceil{i/2}}\cdot x^0
\le \varphi^i\cdot x^0,
\]
where the second step is a consequence of the fact that $\varphi^i=v^{\ceil{i/2}}+\kappa^{-1}u^i$ for $u^i$ a positive vector. It follows that
\begin{align*}
    \varphi^j\cdot z
    &\le c\max\{b,v^{\ceil{i/2}}\cdot x\}
    <\varphi^i\cdot x^0\\
    h_\kappa(z,j)
    &\le 1-\frac1{1+c\max\{b,v^{\ceil{i/2}}\cdot x\}}
    <h_\kappa(x^0,i)
\end{align*}
Therefore, $\sup_{(z,j)\in B} h_\kappa(z,j)<h_\kappa(x^0,i)$. So by Theorem \ref{thm:Lyapunov-transience}, with positive probability the process never returns to $B$ from state $(x^0,i)$. But by construction $(x,i)\in B$, so we see that $(x^0,i)$ is a state reachable from $(x,i)$ from which with positive probability the process never returns to state $(x,i)$. It follows that $(x,i)$ is transient.

To see that $(x,i)$ is actually evanescent, notice that $B$ contains all states of the form $(0,j)$. In particular, every state not in $B$ is transient by above. Therefore, on the event that the process reaches $(x^0,i)$ and then never returns to $B$, the process never enters a recurrent state. By above, this event has positive probability when the process is started from state $(x,i)$, so the state $(x,i)$ is evanescent by Proposition \ref{prop:evanescent}.
\end{proof}

Proposition \ref{prop:ergodic-large-and-small}, in conjunction with Theorems \ref{thm:fast-switching-ergodicity} and \ref{thm:slow-switching-ergodicity}, shows that if we can find $M_1,M_2,M_3,M_4$ all stable such that their partial averages $(M_1+M_2)/2$ and $(M_3+M_4)/2$ are unstable but their full average $(M_1+M_2+M_3+M_4)/4$ is stable again, then we will have a model which is stable for large and small $\kappa$ but unstable for intermediate $\kappa$. As it turns out, such matrices do exist, and they can even be taken to be $2\times 2$.

\begin{example}\label{ex:ergodic-large-and-small}
For some $\kappa,\varepsilon>0$, let $Q$ be the matrix
\[
Q
:=\begin{pmatrix}
    -(1+2\varepsilon) & 1 & \varepsilon & \varepsilon\\
    1 & -1(1+2\varepsilon) & \varepsilon & \varepsilon\\
    \varepsilon & \varepsilon & -(1+2\varepsilon) & 1\\
    \varepsilon & \varepsilon & 1 & -(1+2\varepsilon)\\
\end{pmatrix}
\]
and consider the process randomly switching, with rates given by $\kappa Q$, between the mass-action CRNs
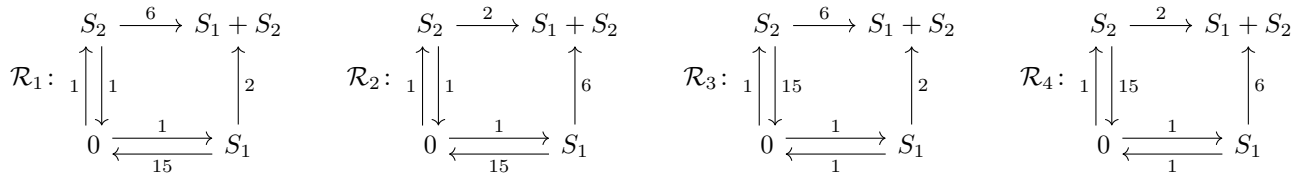
\begin{equation*}
\mathcal{R}_1\colon
\begin{tikzcd}[row sep=3em]
    S_2\arrow[xshift=3]{d}{1}\arrow{r}{6}&S_1+S_2\\
    0\arrow[yshift=3]{r}{1}\arrow[xshift=-3]{u}{1} &  S_1 \arrow[yshift=-3]{l}{15} \arrow[swap]{u}{2}
\end{tikzcd}
\qquad
\mathcal{R}_2\colon
\begin{tikzcd}[row sep=3em]
    S_2\arrow[xshift=3]{d}{1}\arrow{r}{2}&S_1+S_2\\
    0\arrow[yshift=3]{r}{1}\arrow[xshift=-3]{u}{1} &  S_1 \arrow[yshift=-3]{l}{15} \arrow[swap]{u}{6}
\end{tikzcd}
\qquad
\mathcal{R}_3\colon
\begin{tikzcd}[row sep=3em]
    S_2\arrow[xshift=3]{d}{15}\arrow{r}{6}&S_1+S_2\\
    0\arrow[yshift=3]{r}{1}\arrow[xshift=-3]{u}{1} &  S_1 \arrow[yshift=-3]{l}{1} \arrow[swap]{u}{2}
\end{tikzcd}
\qquad
\mathcal{R}_4\colon
\begin{tikzcd}[row sep=3em]
    S_2\arrow[xshift=3]{d}{15}\arrow{r}{2}&S_1+S_2\\
    0\arrow[yshift=3]{r}{1}\arrow[xshift=-3]{u}{1} &  S_1 \arrow[yshift=-3]{l}{1} \arrow[swap]{u}{6}
\end{tikzcd}
\end{equation*}

\end{example}

\begin{cor}
    There exists a choice of $\varepsilon>0$ such that the model from Example \ref{ex:ergodic-large-and-small} is exponentially ergodic for both all $\kappa$ sufficiently large and all $\kappa$ sufficiently small, but such that all states are evanescent for some intermediate choice of $\kappa$.
\end{cor}

\begin{proof}
Notice that the associated matrices here are
\begin{align*}
    M_1=\begin{pmatrix}-15&2\\6&-1\end{pmatrix} &&&&
    M_2=\begin{pmatrix}-15&6\\2&-1\end{pmatrix} &&&&
    M_3=\begin{pmatrix}-1&2\\6&-15\end{pmatrix} &&&&
    M_4=\begin{pmatrix}-1&6\\2&-15\end{pmatrix}.
\end{align*}
Moreover, their averages are
\begin{align*}
    A_1:=&\frac12(M_1+M_2)=\begin{pmatrix}-15&4\\4&-1\end{pmatrix}\\
    A_2:=&\frac12(M_3+M_4)=\begin{pmatrix}-1&4\\4&-15\end{pmatrix}\\
    M=&\frac14(M_1+M_2+M_3+M_4)=\begin{pmatrix}-8&4\\4&-8\end{pmatrix}.
\end{align*}
All seven of these matrices are irreducible, so by Proposition \ref{prop:dec/inc direction} each has a decreasing direction iff it has two eigenvalues with negative real part and has an increasing direction iff it has at least one eigenvalue with positive real part (see Section \ref{sec:lin-alg} for definitions). But $2\times 2$ matrices, like these seven, whose off-diagonal entries have the same sign can only have real eigenvalues, and so their eigenvalues have the same sign iff the determinant is positive and different signs iff the determinant is negative. But all seven matrices here have negative trace and hence at least one negative eigenvalue, so combining everything in this paragraph we conclude that each matrix has a decreasing direction iff it has positive determinant and has an increasing direction iff it has negative determinant. But $M_1,M_2,M_3,M_4,M$ all can readily be seen to have positive determinant, and $A_1,A_2$ have negative determinant, so we conclude that the former each have a decreasing direction and the latter each have an increasing direction.

Because $M_1,M_2,M_3,M_4$ each has a decreasing direction, Theorem \ref{thm:slow-switching-ergodicity} tells us that for every $\varepsilon$, it is the case that the system is exponentially ergodic for all sufficiently small $\kappa$ (with the threshold for ``sufficiently small" depending, of course, on $\varepsilon$). Similarly, because $M$ has a decreasing direction, Theorem \ref{thm:fast-switching-ergodicity} tells us that for every $\varepsilon$ the system is exponentially ergodic for sufficiently large $\kappa$ (again, depending on $\varepsilon$). Therefore, we just have to show that for some choice of $\varepsilon$ and $\kappa$ every state is evanescent for the system. But since $A_1$ and $A_2$ each has an increasing direction, Proposition \ref{prop:ergodic-large-and-small} tells us that there exists $\varepsilon,\kappa$ such that every state with at least one of either species is evanescent; the existence of inflow reactions means that, in fact, every state is evanescent for such $\varepsilon,\kappa$.
\end{proof}

We now make use of Example \ref{ex:ergodic-large-and-small} by generalizing the proof of Proposition \ref{prop:trans-large-and-small} in order to show that there exist models which undergo arbitrarily many phase transitions from exponentially ergodic to transient and back as $\kappa$ increases, and that moreover such models can be taken to have our preferred behavior for $\kappa$ sufficiently large and sufficiently small.

\begin{prop}\label{prop:arbitrarily-many-phase-transitions}
    Let $N\in\ZZ_{>0}$. Then there exists a model in the \nameref{setting} Setting with mass-action kinetics which undergoes at least $N$ transitions from exponentially ergodic to evanescent and back as $\kappa$ increases. Moreover, we may choose this model to be either evanescent or exponentially ergodic for all sufficiently small $\kappa$, and independent of that choice we may choose the model to be either evanescent or exponentially ergodic for all sufficiently large $\kappa$.
\end{prop}

\begin{proof}
We will consider a model in the \nameref{setting} Setting with four environments. Let $Q$ be the matrix
\[
Q
:=\begin{pmatrix}
    -(1+2\varepsilon) & 1 & \varepsilon & \varepsilon\\
    1 & -1(1+2\varepsilon) & \varepsilon & \varepsilon\\
    \varepsilon & \varepsilon & -(1+2\varepsilon) & 1\\
    \varepsilon & \varepsilon & 1 & -(1+2\varepsilon)\\
\end{pmatrix}
\]
where $\varepsilon$ is some fixed number such that the model from Example \ref{ex:ergodic-large-and-small} is transient for at least one value of $\kappa$. Specifically, we consider a model which transitions between environments with rate $\kappa Q$ and which has environment matrices $M_1,M_2,M_3,M_4$ where for each $i$, the matrix $M_i$ has the block diagonal form
\begin{align*}
    M_i=
    \begin{pmatrix}
    M^0_i   &   0               &   0               &   \cdots  &   0\\
    0       &   \beta_1 M^1_i   &   0               &   \cdots  &   0\\
    0       &   0               &   \beta_2 M^2_i   &   \cdots  &   0\\
    \vdots  &   \vdots          &   \vdots          &   \ddots  &   \vdots\\
    0       &   0               &   0           &   \cdots  &   \beta_{N+1}M^{N+1}_i
    \end{pmatrix}
\end{align*}
where the $M^j_i$ are square matrices which we will describe momentarily, the $0$s are (possibly non-square) matrices of various sizes consisting only of zeros, and $\beta_1,\cdots,\beta_{N+1}$ are positive real numbers to be chosen later.

If we wish to make our model evanescent for all sufficiently small $\kappa$, let $M^0_1,M^0_2,M^0_3,M^0_4$ be matrices associated to a $4$-environment \nameref{setting} Setting model with environmental transition rate matrix $\kappa Q$ which is evanescent for sufficiently small $\kappa$ and exponentially ergodic for sufficiently large $\kappa$ (for example, you could take $M^0_1=M^0_2=M_1$ and $M^0_3=M^0_4=M_2$ where $M_1,M_2$ are the matrices from Example \ref{ex:switching ergodic monomolecular}; the result is effectively a $2$-environment system which transitions with rate $2\kappa\varepsilon$ in each direction). If instead we wish to make our model exponentially ergodic for all sufficiently small $\kappa$, let $M^0_1,M^0_2,M^0_3,M^0_4$ be matrices associated to a $4$-environment \nameref{setting} Setting model with environmental transition rate matrix $\kappa Q$ which is exponentially ergodic for all $\kappa$ (for example, you could take $M^0_1=M^0_2=M^0_3=M^0_4$ to be any stable matrix). In either case, let $\kappa^0_{\min}\in[0,\infty)$ be large enough that the model associated to the $M^0_i$ is exponentially ergodic whenever $\kappa>\kappa^0_{\min}$.

Similarly, if we wish to make our model evanescent for all sufficiently large $\kappa$, let $M^{N+1}_i$ be matrices associated to a $4$-environment \nameref{setting} Setting model with environmental transition rate matrix $\kappa Q$ which is evanescent for sufficiently large $\kappa$ and exponentially ergodic for sufficiently small $\kappa$ (like Example \ref{ex:basic switching transient}). And if instead we wish to make our model exponentially ergodic for all sufficiently large $\kappa$, let the $M^{N+1}_i$ be matrices associated to a $4$-environment \nameref{setting} Setting model with environmental transition rate matrix $\kappa Q$ which is exponentially ergodic for all $\kappa$. In either case, let $\kappa^{N+1}_{\max}\in(0,\infty]$ be small enough that the model associated with the $M^{N+1}_i$ is exponentially ergodic whenever $\kappa<\kappa^{N+1}_{\max}$. Note that the model associated to $\beta_{N+1}M^{N+1}_1,\cdots,\beta_{N+1}M^{N+1}_4$ is exponentially ergodic whenever $\kappa<\beta_{N+1}\kappa^{N+1}_{\max}$ (for more detail on this, see the proof of Proposition \ref{prop:trans-large-and-small}, and especially Figure \ref{fig:rescaling}).

Lastly, for each $j=1,\cdots,N$, let $M^j_i$ be matrices associated to a $4$-environment \nameref{setting} Setting model with environmental transition rate matrix $\kappa Q$ which is exponentially ergodic both for sufficiently small $\kappa$ and for sufficiently large $\kappa$, but is evanescent for at least one intermediate $\kappa$ (like Example \ref{ex:ergodic-large-and-small}, for instance). Let $\kappa^j_{\max},\kappa^j_{\min}\in(0,\infty)$ be such that $\kappa^j_{\max}<\kappa^j_{\min}$, and the model associated with the $M^j_i$ is exponentially ergodic whenever either $\kappa<\kappa^j_{\max}$ or $\kappa>\kappa^j_{\min}$. Note that the model associated to the $\beta_jM^j_i$ is exponentially ergodic whenever either $\kappa<\beta_j\kappa^j_{\max}$ or $\kappa>\beta_j\kappa^j_{\min}$, but will be evanescent for at least one $\kappa\in(\beta_j\kappa^j_{\max},\beta_j\kappa^j_{\min})$.

For ease of notation set $\beta_0:=1$, and for $j=1,\cdots,N+1$ recursively fix $\beta_j$ large enough that $\beta_{j-1}\kappa^{j-1}_{\min}<\beta_j\kappa^j_{\max}$. Notice that for each $j\in\{1,\cdots,N+1\}$ and each $\kappa\in(\beta_{j-1}\kappa^{j-1}_{\min},\beta_j\kappa^j_{\max})$, each of the $N+2$ submodels is exponentially ergodic, and hence so is the whole model. However, by construction for each $j\in\{1,\cdots,N\}$ there exists a $\kappa\in(\beta_j\kappa^j_{\max},\beta_j\kappa^j_{\min})$ such that the $j$-th submodel is evanescent, and thus for such $\kappa$ the whole model is evanescent. It follows that the model undergoes at least $N$ transitions from exponentially ergodic to evanescent and back as $\kappa$ increases. To see that the model has the correct behavior for $\kappa$ very small, note that all submodels but the first are exponentially ergodic for sufficiently small $\kappa$, and hence the whole model will have the same stability behavior for small $\kappa$ as the first submodel, which recall was chosen to be either exponentially ergodic or evanescent as desired. The verification for large $\kappa$ is the same.
\end{proof}

\section{Open Problems}\label{sec:Open Problems}

Recall that Theorem \ref{thm:slow-switching-transience} gives evanescence in the case where $\kappa$ is small and there exists vectors $v^i$ which satisfy certain hypotheses, one of which is that $\supp(v^i)$ does not depend on $i$. Example \ref{ex:disjoint species} demonstrates that we cannot simply drop this last hypothesis. Indeed, in that example the two environments have vectors $v^i$ which satisfy all other hypothesis of Theorem \ref{thm:slow-switching-transience}, but $\supp(v^1)\cap\supp(v^2)=\emptyset$ and the conclusion of the theorem is not satisfied. That said, it is possible to imagine a middle ground between all the supports being the same, and them being collectively disjoint. Indeed, it is consistent with Example \ref{ex:disjoint species} to conjecture the following:

\begin{conj}\label{conj: general conjecture}
    In the \nameref{setting} Setting, assume additionally that each CRN is endowed with mass-action kinetics and has at-most-monomolecular reactions. Suppose for each $i$ there exists a vector $v^i\in\RR^d_{\ge0}$ such that $(v^iM_i)_m>0$ for each $m\in\supp(v^i)$, where $(v^iM_i)_m$ is the $m$-th entry of the vector $v^iM_i$. Suppose further that $\bigcap_j \supp(v^j)\ne\emptyset$. Then as long as $\kappa>0$ is sufficiently small, $(x,i)$ is an evanescent state for the mixed Markov chain $Z_\kappa$ whenever $\left(\bigcap_j\supp(v^j)\right)\cap\supp(x)\ne\emptyset$.
\end{conj}

We are not able to prove or disprove this conjecture. Moreover, the mixed process in the following example would be transient for small $\kappa$ if the conjecture were true, but even for this example we are not able to determine transience for small $\kappa$.

\begin{example}\label{ex:IDK man}
Let $\mathcal R_1$ and $\mathcal R_2$ be the mass-action CRNs described below, for some $\alpha>0$ to be chosen later:
    \begin{center}
    $\mathcal R_1$:
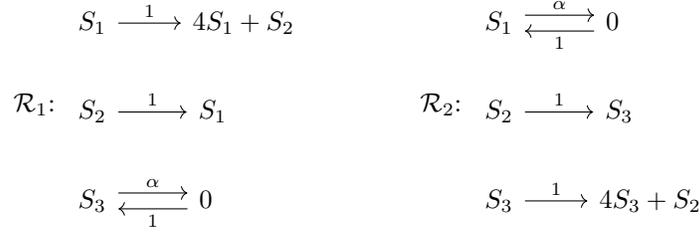
    \begin{tikzcd}
        S_1\arrow{r}{1}         &   4S_1+S_2\\
        S_2\arrow{r}{1}         &   S_1\phantom{+4S_2}\\
        S_3\arrow[yshift=3]{r}{\alpha}    &   0\phantom{+42S_2}\arrow[yshift=-3]{l}{1}
    \end{tikzcd}
    \qquad\qquad$\mathcal R_2$:
    \begin{tikzcd}
        S_1\arrow[yshift=3]{r}{\alpha}    &   0\phantom{+42S_2}\arrow[yshift=-3]{l}{1}\\
        S_2\arrow{r}{1}         &   S_3\phantom{+4S_2}\\
        S_3\arrow{r}{1}         &   4S_3+S_2
    \end{tikzcd}
    \end{center}
Consider the mixed process $Z_\kappa$ which transitions between environments 1 and 2 with rate $\kappa$ in each direction, and evolves according to $\mathcal R_i$ in environment $i$. Note that $Z_{\kappa}$ is irreducible for all positive $\kappa$.
\end{example}

\begin{prop}
In Example \ref{ex:IDK man}, if Conjecture \ref{conj: general conjecture} holds then for each fixed value of $\alpha>0$, for all small enough values of $\kappa$ every state is evanescent for $Z_\kappa$.
\begin{proof}
    The matrices associated to $\mathcal R_1$ and $\mathcal R_2$ are given by
    \begin{align*}
        M_1=\begin{pmatrix}3&1&0\\1&-1&0\\0&0&-\alpha\end{pmatrix}
        \qquad
        M_2=\begin{pmatrix}-\alpha&0&0\\0&-1&1\\0&1&3\end{pmatrix},
    \end{align*}
    respectively. If $v^1:=(2,1,0)$ and $v^2:=(0,1,2)$, then one has $v^1M_1=(7,1,0)$ and $v^2M_2=(0,1,7)$. Therefore, it would follow from Conjecture \ref{conj: general conjecture} that for all small enough values of $\kappa$, all states with at least one molecule of $S_2$ were evanescent. But the existence of an inflow reaction for $S_1$ as well as a reaction taking $S_1$ as input and producing $S_2$ means that from every state, it is possible to reach another state with at least one $S_2$. Thus given Conjecture \ref{conj: general conjecture} every state would be evanescent for small $\kappa$, as claimed.
\end{proof}
\end{prop}

A special case, then, of Conjecture \ref{conj: general conjecture} which is perhaps more tractable is the following:

\begin{conj}\label{conj: specific conjecture}
    There exists $\alpha>0$ such that the mixed process from Example \ref{ex:IDK man} is transient for all small enough values of $\kappa$.
\end{conj}

It is worth noting that Conjecture \ref{conj: general conjecture} cannot be proven using the Lyapunov function that was used in the proof of the theorem which it generalizes (Theorem \ref{thm:slow-switching-transience}). Indeed, the proof of that theorem proceed by defining the function $h(x,i)=1-(1+v^i\cdot x)^{-1}$ and showing that $\mathcal L_\kappa h(x,i)$ was positive as long as $v^i\cdot x$ was sufficiently large. In Proposition \ref{prop:function doesn't work}, by contrast, we will show in the context of Example \ref{ex:IDK man} that for all $b$, there exists a state $(x,i)$ such that $v^i\cdot x>b$ but nevertheless $\mathcal L_\kappa h(x,i)<0$. This suggests that even if Conjecture \ref{conj: general conjecture} is true, techniques beyond those used in this paper will be needed to prove it.

\begin{prop}\label{prop:function doesn't work}
    There exists a process in the \nameref{setting} Setting such that additionally:
    \begin{itemize}
        \item each constituent CRN of the process is endowed with mass-action kinetics and has at-most-monomolecular reactions,

        \item for each $i$ there exists a vector $v^i\in\RR^d_{\ge0}$ such that $(v^iM_i)_m>0$ for each $m\in\supp(v^i)$, where $(v^iM_i)_m$ is the $m$-th entry of the vector $v^iM_i$,

        \item $\bigcap_j \supp(v_j)\ne\emptyset$,
    \end{itemize}
    but nevertheless, if $h$ is the function $h(x,i):=1-(1+v^i\cdot x)^{-1}$ then for any choices of $\kappa>0$ and $b$ there exists a state $(x,i)$ such that $v^i\cdot x>b$ but $\mathcal L_\kappa(x,i)<0$.
\begin{proof}
    Consider the mixed process from Example \ref{ex:IDK man}. Notice that if $v^1:=(2,1,0)$ and $v^2:=(0,1,2)$, then one has $v^1M_1=(7,1,0)$ and $v^2M_2=(0,1,7)$. Let $h(x,i)=1-(1+v^i\cdot x)^{-1}$, and notice that
    \begin{align*}
        \mathcal L_\kappa(x,1)
        &=\frac1{1+2x_1+x_2}\left(\frac{\kappa(-2x_1+2x_3)}{1+x_2+2x_3}+\frac{7x_1}{1+2x_1+x_2+7}+\frac{x_2}{1+2x_1+x_2+1}+\frac{\alpha x_3(0)}{1+2x_1+x_2}\right)\\
        &=\frac1{1+2x_1+x_2}\left(\frac{2\kappa(x_3-x_1)}{x_2+2x_3+1}+\frac{7x_1}{2x_1+x_2+8}+\frac{x_2}{2x_1+x_2+2}\right)\\
        &\le\frac1{1+2x_1+x_2}\left(\kappa-\frac{2\kappa x_1}{x_2+2x_3+1}+\frac{7}{2}+1\right).
    \end{align*}
    This upper bound is evidently negative iff the expression inside the parentheses is negative. Moreover, for any fixed values of $\kappa$, $x_2$, and $x_3$, the expression will be negative for all sufficiently large values of $x_1$. The result follows.
\end{proof}
\end{prop}

\section*{Appendix}

The section is for results and definitions which are worthy of inclusion but not central for this paper.

\subsection{Lyapunov function results}

For convenience, we state the following result from \cite{Meyn_Tweedie_1993}.

\begin{theorem*}[Theorem 7.1 in \cite{Meyn_Tweedie_1993}]
    Let $X$ be an irreducible continuous-time Markov chain on a countable discrete state space $\mathbb S$ with generator $\mathcal L$. Suppose that there exists a function $V:\mathbb S\to[0,\infty)$ and constants $c\in(0,\infty)$ and $d\in\RR$ such that $\{x\in\mathbb{S}\,:\,V(x)\leq R\}$ is finite for all $R\in [0,\infty)$ and
    \[
    \mathcal LV(x)\le -cV(x)+d
    \]
    for every $x\in\mathbb S$. Then $X$ is exponentially ergodic.
\end{theorem*}

We now state a closely related result for continuous-time Markov chains that are not necessarily irreducible. We could not find a proof for this formulation in the literature, so we give one here for completeness. The proof could probably be adapted from that of Theorem 7.1 in \cite{Meyn_Tweedie_1993}, but we propose a shorter proof that makes use of Theorem 7.1 in \cite{Meyn_Tweedie_1993} instead.

\begin{theorem}\label{thm:Lyapunov-exponential-ergodicity}
    Let $X$ be a continuous-time Markov chain on a countable discrete state space $\mathbb S$ with generator $\mathcal L$. Suppose that there exists a function $V:\mathbb S\to[0,\infty)$ and constants $c\in(0,\infty)$ and $d\in\RR$ such that $\{x\in\mathbb{S}\,:\,V(x)\leq R\}$ is finite for all $R\in [0,\infty)$ and
    \[
    \mathcal LV(x)\le -cV(x)+d
    \]
    for every $x\in\mathbb S$. Then $X$ converges exponentially fast.
\begin{proof}
    The continuous-time Markov chain restricted to a closed communicating class can be regarded as an irreducible Markov chain. Hence, $X$ restricted to a closed communicating class is exponentially ergodic by Theorem~7.1 in \cite{Meyn_Tweedie_1993}. Let $A$ be the union of closed communicating classes and let $\tau=\inf\{t\in[0,\infty)\,:\,X(t)\in A\}$. By Theorem~7.3 in \cite{Anderson_Cappelletti_Kim_Nguyen_2020} we have that $\mathbb{E}_x[\tau]<\infty$ for all $x\in\mathbb{S}$. Let $Q$ be the transition rate matrix of $X$ and let $\mu$ be a probability measure on $\mathbb{S}\setminus A$. Define a continuous-time Markov chain $\hat{X}$ on $\mathbb{S}$ with transition rate matrix $\hat{Q}$ given by
    \[
    \hat{Q}(i,j)=\begin{cases}
    \mu(j)\min\left\{1, \dfrac{1}{V(j)}, \dfrac{1}{\mathbb{E}_j[\tau]}\right\}&\text{if }i\in A, j\notin A\\
    Q(i,j) &\text{otherwise}
    \end{cases}
    \]
    for all $i,j\in\mathbb{S}$ with $i\neq j$, where we set $1/0=\infty$. Since $\mathbb{E}_x[\tau]<\infty$ for all $x\in\mathbb{S}$ and there is a positive transition rate from any state $i\in A$ to any state $j\notin A$, the Markov chain $\hat{X}$ is irreducible. Moreover, we can couple the two processes $X$ and $\hat{X}$ in a way that $X(t)=\hat{X}(t)$ for all $t\leq \tau$. With this choice of coupling, we can write $\tau=\inf\{t\in[0,\infty)\,:\,\hat{X}(t)\in A\}$. Finally, let $\hat{\mathcal{L}}$ be the generator of $\hat{X}$. For all $x\notin A$ we have $\hat{\mathcal{L}} V(x)=\mathcal{L}V(x)\leq -cV(x)+d$, and for all $x\in A$
    \begin{align*}
        \hat{\mathcal{L}} V(x)&=\mathcal{L}V(x)+\sum_{j\notin A}\mu(j)\min\left\{1, \dfrac{1}{V(j)}, \dfrac{1}{\mathbb{E}_j[\tau]}\right\}(V(j)-V(x))\\
        &\leq \mathcal{L}V(x)+\sum_{j\notin A}\mu(j)\dfrac{V(j)}{V(j)}-V(x)\sum_{j\notin A}\mu(j)\min\left\{1, \dfrac{1}{V(j)}, \dfrac{1}{\mathbb{E}_j[\tau]}\right\}\\
        &\leq \mathcal{L}V(x)+1\leq -cV(x)+d+1.
    \end{align*}
    Hence, Theorem~7.1 in \cite{Meyn_Tweedie_1993} applies and we can conclude that $\hat{X}$ is exponentially ergodic. Hence, due to \cite[Proposition 12]{kingman1964stochastic}, stated as Theorem 1.1 and rigorously proved in \cite{kim2025path}, there exists $\gamma\in\RR_{>0}$ such that for all $x\in\mathbb S$ we have $\E_x[e^{\gamma\tau_x}]<\infty$, where
    \begin{equation*}
        \tau_x=\inf\{t\in\RR_{>0}\,:\,\hat{X}(t)=x, \hat{X}(s)\neq x \text{ for some } 0\leq s<t\}.
    \end{equation*}
   Let $y\in A$. By irreducibility, given any $x\in \mathbb S$ we have $P_y(\tau_x<\tau_y)>0$. Hence, by strong Markov property
   \begin{align*}
       \infty>\E_y[e^{\gamma\tau_y}]&\geq P_y(\tau_x<\tau_y)\;\E_y[e^{\gamma\tau_x}|\tau_x<\tau_y]\;\E_x[e^{\gamma\tau_y}]\\
       &\geq P_y(\tau_x<\tau_y)\;\E_y[e^{\gamma\tau_x}|\tau_x<\tau_y]\;\E_x[e^{\gamma\tau}]
   \end{align*}
   It follows that $\E_x[e^{\gamma\tau}]<\infty$ for every $x\in \mathbb S$, and the proof is complete.
\end{proof}
\end{theorem}

Above, we made use of the following result to prove transience. For a proof of the version given here, see for instance the appendix of \cite{Anderson_Howells_2023}.

\begin{theorem}\label{thm:Lyapunov-transience}
    Let $X$ be a non-explosive continuous-time Markov chain on a countable discrete state space $\mathbb S$ with generator $\mathcal L$. Let $B\subset\mathbb S$, and let $\tau_B$ be the time for the process to enter $B$. Suppose there is some bounded function $V$ such that for all $x\in B^c$,
    \[
    \mathcal LV(x)\ge0.
    \]
    Then $\PP_{x_0}(\tau_B<\infty)<1$ for any $x_0$ such that
    \[
    \sup_{x\in B}V(x)<V(x_0).
    \]
\end{theorem}

\subsection{Linear algebra results}\label{sec:lin-alg}

This section summarizes some definitions and results from linear algebra. No originality is claimed. Our goal is to allow the hypotheses of the main theorems of this paper to be understood in terms that may be more familiar. Specifically, Proposition \ref{prop:dec/inc direction}(i) tells us that a hypothesis of Theorem \ref{thm:fast-switching-ergodicity} (namely, that there exists a $v\in\RR^d_{>0}$ with $vM\in\RR^d_{<0}$) is equivalent to the matrix $M$ being Hurwitz stable (having only eigenvalues with negative real part), and similarly for the matrices $M_i$ in Theorem \ref{thm:slow-switching-ergodicity}. Meanwhile, Proposition \ref{prop:Hurwitz-unstable} tells us that the similar conditions in Theorems \ref{thm:fast-switching-transience} and \ref{thm:slow-switching-transience} are equivalent to the matrices $M$ and $M_i$, respectively, being Hurwitz unstable (having at least one eigenvalue with positive real part).

Let $d\in\ZZ_{>0}$, let $M=(m_{ij})$ be a $d\times d$ matrix and $I\subseteq[d]$ be non-empty. We assume the set $I$ inherits the ordering from $[d]$. Let $M^I$ denote the $|I|\times|I|$ matrix consisting of $m_{ij}$ for $i,j\in I$. We refer to $M^I$ as the \emph{$I$-th principal submatrix of $M$}, or simply as a \emph{principal submatrix} of $M$. Similarly, for $v=(v_1,\cdots,v_d)\in\RR^d$ let $v^I$ be the vector of length $|I|$ consisting of $v_i$ for $i\in I$.

A square matrix $P$ is a \emph{permutation matrix} if exactly one entry in each row of $P$ is $1$, exactly one entry in each column of $P$ is $1$, and all other entries of $P$ are zero. A square matrix $M$ is \emph{reducible} if there exists a permutation matrix $P$ such that
\begin{align*}
    P^\intercal MP=\begin{bmatrix}
        M_{11} & M_{12}\\ 0 & M_{22}
    \end{bmatrix},
\end{align*}
where $M_{11}$ and $M_{22}$ are square matrices and $0$ denotes a (possibly non-square) matrix of all zeros. A square matrix which is not reducible is call \emph{irreducible}. With a suitable permutation any square matrix can be written as a block-upper-triangular matrix where the diagonal matrices are irreducible; that is, for any $M$ there exists a permutation matrix $P$ such that
\begin{align*}
    P^\intercal MP=\begin{bmatrix}
        M_{11} & M_{12} & \cdots & M_{1k}\\
        0 & M_{22} & \cdots & M_{2k}\\
        \vdots & \vdots & \ddots & \vdots\\
        0 & 0 & \cdots & M_{kk}
    \end{bmatrix},
\end{align*}
where each $M_{ii}$ is irreducible. This expression is known as the \emph{irreducible normal form} or \emph{Frobenius normal form} (see \cite[Chapter 8, Section 3, Problem 8]{Horn_Johnson_2013}) of $M$. The Frobenius normal form is unique up to permutations of the blocks and permutations within the blocks. For the existence and uniqueness of Frobenius normal form, see e.g.~\cite[chapter XIII, Sec. 4]{Gantmacher_1959}. Note that each $M_{ii}$ is a (permutation of a) principal submatrix of $M$; if $M^I$ is a principal submatrix of $M$ which can be written as such an $M_{ii}$ then we say $M^I$ is a \emph{Frobenius principal submatrix}, or sometimes a \emph{maximal irreducible principal submatrix}.

A matrix $M=(m_{ij})$ is said to be \emph{Metzler} if $m_{ij}\ge0$ whenever $i\ne j$. A matrix $M$ (not necessarily Metzler) is said to be \emph{Hurwitz stable} if every eigenvalue of $M$ has strictly negative real part; on the other hand, if there exists an eigenvalue of $M$ with strictly positive real part then $M$ is said to be \emph{Hurwitz unstable}.

\begin{remark}\label{rmk:passing to submatrices}
Notice that passing to a principle submatrix preserves the Metzler property. In addition, note that because the Frobenius normal form is block-triangular, the eigenvalues of $M$ will be exactly the collective eigenvaules of the Frobenius principal submatrices of $M$. In particular, $M$ is Hurwitz stable iff every Frobenius principal submatrix of $M$ is Hurwitz stable, and $M$ is Hurwitz unstable iff some Frobenius principal submatrix of $M$ is Hurwitz unstable.
%
\hfill //
\end{remark}

A $d\times d$ matrix $M$ \emph{has an increasing direction} if there exists a $w\in\RR^d_{>0}$ such that $wM\in\RR^d_{>0}$, and such a $w$ is called an \emph{increasing direction for $M$}. Similarly, if there is a $w\in\RR^d_{>0}$ such that $wM\in\RR^d_{<0}$, we say $w$ is a \emph{decreasing direction for $M$} and that $M$ \emph{has a decreasing direction}.

The following proposition characterizes the matrices which have increasing and decreasing directions. Part (i) gives an alternate formulation of the hypotheses of Theorems \ref{thm:fast-switching-ergodicity} and \ref{thm:slow-switching-ergodicity}. Part (ii) will be used in the proof of Proposition \ref{prop:Hurwitz-unstable}.

\begin{prop}\label{prop:dec/inc direction}
Let $M$ be a Metzler matrix.
\begin{itemize}
    \item[(i)] $M$ has decreasing direction iff $M$ is Hurwitz stable iff all Frobenius principal submatrices of $M$ are Hurwitz stable.

    \item[(ii)] $M$ has increasing direction iff all Frobenius principal submatrices of $M$ are Hurwitz unstable.
\end{itemize}
\end{prop}

\begin{remark}
Proposition \ref{prop:dec/inc direction}(i) is well-known (see the references provided below). We believe that (ii) is also well-known, but we lack a reference and therefore have included a proof.\hfill //
\end{remark}

\begin{proof}[Proof of Proposition \ref{prop:dec/inc direction}]
    The first biimplication in (i) follows from \cite[Theorem 2.3($I_{28}$)]{Berman_Plemmons_1994}; see also \cite[Theorem 2.5.3 (2.5.3.12)]{Horn_Johnson_1991}. The second biimplication in (i) was discussed in Remark \ref{rmk:passing to submatrices}.

    For (ii),  Let $\hat{A}$ be a normal form of $A$: $A=P \hat{A} P^T$ for some permutation matrix $P$.  If there exists $\hat{w}>0$ such that $\hat{w}^T\hat{A}>0$, let $w=P\hat{w}$. Then we have $w>0$ and $w^TA=\hat{w}^TP^TP \hat{A} P^T= \hat{w}^T \hat{A}P^T>0$ (note $\hat{w}^T\hat{A}>0$ and hence reordering the coordinates of a positive vector by a permutation matrix $P^T$ will preserve positivity of the vector). Hence w.l.o.g. assume $A$ is in a normal form. Since principal submatrices $A_{ii}$ are irreducible and Metzler, we have $s(A_{ii})I+A_{ii}$ is nonnegative and hence $\rho(s(A_{ii})I+A_{ii})=r(s(A_{ii})I+A_{ii})=s(A_{ii})+r(A_{ii})$ \cite[Sec.8.3, Problem 9]{Horn_Johnson_2013}. 

    $\Leftarrow$. Since $r(A_{ii})>0$ for all $i=1,\ldots,k$,  note $s(A_{ii})I+A_{ii}$ is nonnegative and irreducible, by Perron-Frobenius Theorem \cite[Theorem~8.4.4]{Horn_Johnson_2013}, there exists a unique (up to a scalar) positive left Perron eigenvector $w_i$ such that
    $$
    w_i^T(s(A_{ii})I+A_{ii})=\rho(s(A_{ii})I+A_{ii})w_i^T=(s(A_{ii})+r(A_{ii}))w_i^T,
    $$
    which implies that $w_i^TA_{ii}=r(A_{ii})w_i^T$. Let  $w=(w_1,\gamma w_2,\cdots,\gamma^{k-1}w_k)^T$. In the light of that the sign of a polynomial for all large values is determined by the sign of the coefficient of its monomial  of highest degree, one can show for all sufficiently large $\gamma>0$,
    $$
    w^TA=\left[w_1^TA_{11},w_1^TA_{12}+\gamma w_2^TA_{22},\cdots,\sum_{j=1}^{\ell}\gamma^{j-1}w_j^TA_{jj},\cdots,\sum_{j=1}^k\gamma^{k-1}w_j^TA_{jj}\right]
    $$
    is positive.
    
    $\Rightarrow$. We prove by contraposition. Suppose $r(A_{ii})\le0$ for some $1\le i\le k$. If $r(A_{ii})<0$, then $A_{ii}$ is a Hurwitz stable Metzler matrix. Assume w.o.l.g. that $i=1$. Otherwise, we choose a different permutation so that the first principal submatrix of $A$ is Hurwitz stable. Let $A_{11}\in\mathcal{M}^+_{n_1\times n_1}(\mathbb{R})$. If there exists $w=(w_1,\ldots,w_k)^T>0$ such that $w^TA>0$, then $w_1^TA_{11}>0$. Consider the $n_1$ dimensional linear ODE $x'(t)=A_{11}x$. Then by linear stability theory of ODE, $r(A_{11})<0$ implies that $\lim_{t\to\infty}x(t)=0$. Now consider $V(x)=(w_1^T\cdot x)^2$. Then $\frac{d}{dt}V(x(t))=2x^Tw_1w_1^TA_{11}x(t)\ge c V(x(t))$ for all $x(t)\in\mathbb{R}^{n_1}_+\setminus\{0\}$, where $c=2\min_{i=1}^{n_1}\frac{(w_1^TA_{11})_i}{(w_1^T)_i}>0$. Hence it implies that $V(x(t))\ge V(x(0))e^{ct}$. This shows  that $\lim_{t\to\infty}\|x(t)\|_1=\infty$ for all $x(0)\in\mathbb{R}^{n_1}_+\setminus\{0\}$. The contradiction  implies that $r(A_{jj})\ge0$ for all $j=1,\ldots,k$.  Hence $r(A_{ii})=0$ for some $1\le i\le k$. Again assume w.o.l.g. that $i=1$. Then $A_{11}-\varepsilon I_{n_1}$ is Hurwitz stable and Metzler for any $\varepsilon>0$, where $I_{n_1}$ is the $n_1$ by $n_1$ identity matrix. If there exists $w=(w_1,\ldots,w_k)^T>0$ such that $w^TA>0$, then $w_1^TA_{11}>0$. Choose $\varepsilon>0$ such that $w_1-\varepsilon\mathbf{1}>0$ and $w_1^TA_{11}-\varepsilon\mathbf{1}^T>0$, where $\mathbf{1}=(1,\ldots,1)\in\mathbb{R}^{n_1}$. Applying the same argument to the Hurwitz stable Metzler matrix $A_{11}-\varepsilon I_{n_1}$ also yields a contradiction. In sum, $r(A_{ii})>0$ for all $1\le i\le k$.
\end{proof}

Lastly, a matrix $N$ is said to be \emph{non-negative} if all entries of $N$ are non-negative. By the Perron--Frobenius theorem, any non-negative matrix has a non-negative real eigenvalue which is at least as large as the modulus of each other eigenvalue of $N$. This eigenvalue will necessarily have the largest real part of any eigenvalue of $N$. We will abuse terminology slightly and refer to this (not necessarily simple) eigenvalue as \emph{the largest eigenvalue of $N$}.

Notice that if $M\in\RR^{d\times d}$ is a Metzler matrix, then for any large enough $\alpha\in\RR$ we have that $N:=M+\alpha I_d$ is non-negative, where $I_d$ denotes the $d\times d$ identity matrix. Since the eigenvalues of $M+\alpha I_d$ are the eigenvalues of $M$ shifted by $\alpha$, we see that $M$, like $N$, will have an a real eigenvalue with the largest real part of any eigenvalue of $M$. As above, we will refer to this eigenvalue as \emph{the largest eigenvalue of $M$}.

\begin{lemma}\label{lem:Metzler largest eigenvalue}
Suppose that $M$ is a $d\times d$ Metzler matrix. Let $r$ denote the largest eigenvalue of $M$. Let $r'$ denote the maximum over $\emptyset\ne I\subsetneq[d]$ of the largest eigenvalue of the (Metzler) matrix $M^I$. Then $r'\le r$.
\begin{proof}
    This lemma, but for specifically non-negative matrices rather than general Metzler matrices, is Proposition 4 in \cite[Chapter XIII, Section 3]{Gantmacher_1959}. To extend to Metzler matrices $M$, let $N:=M+\alpha I_d$ for some $\alpha\in\RR_{>0}$ no smaller in absolute value than all diagonal entries of $M$, so that $N$ is non-negative. Then the largest eigenvalue of $N$ is $r+\alpha$. Moreover, for each $\emptyset\ne I\subseteq[d]$ the largest eigenvalue of $N^I$ is $\alpha$ plus the largest eigenvalue of $M^I$. It follows that the maximum over $I$ of the largest eigenvalue of $N^I$ is $r'+\alpha$. Since the lemma holds for non-negative matrices we have $r'+\alpha\le r+\alpha$; the lemma in full generality follows.
\end{proof}
\end{lemma}

The following proposition is one of the main results of this subsection. Condition (i) below appears in the hypotheses of Theorems \ref{thm:fast-switching-transience} and \ref{thm:slow-switching-transience}. Conditions (ii), (iii), and especially (iv) provide an alternate way of viewing condition (i).

\begin{prop}\label{prop:Hurwitz-unstable}
Let $M=(m_{ij})$ be a $d\times d$ Metzler matrix. Then the following are equivalent:
\begin{enumerate}
    \item[(i)] There exists some nonzero $v\in\RR^d_{\ge0}$ such that $(vM)_i$, the $i$-th entry of the vector $vM$, is strictly positive whenever $v_i>0$.
    
    \item[(ii)] There exists a principal submatrix of $M$ which has an increasing direction.
    
    \item[(iii)] There exists a principal submatrix of $M$ which is Hurwitz-unstable.

    \item[(iv)] $M$ is Hurwitz-unstable.
\end{enumerate}
\end{prop}
\begin{proof}
Suppose that $I\subseteq[d]$ is nonempty, and suppose $v\in\RR^d_{\ge0}$ satisfies $I=\supp(v)$. Then for all $j\in I$,
\begin{align}\label{eq:(vM)_j}
    (vM)_j
    =\sum_{i=1}^d v_i m_{ij}
    =\sum_{i\in I}v_i m_{ij}
    =(v^IM^I)_j.
\end{align}

\emph{(i) $\implies$ (ii):} Suppose condition (i) is satisfied, and let $v$ be the vector given by the condition. Define $I:=\supp(v)$, and let $w=v^I$. Then $I$ is nonempty since $v$ is nonzero by assumption, and by assumption on $vM$ and the definition of $I$ we have that $(vM)_j>0$ whenever $j\in I$. But it then follows from equation \eqref{eq:(vM)_j} that the vector $wM^I$ is strictly positive, so $w$ is an increasing direction for $M^I$.

\emph{(ii) $\implies$ (i):} Suppose that $I$ is nonempty and that $w$ is an increasing direction for $M^I$. Define $v\in\RR^d_{\ge0}$ via $v_i=w_i$ for $i\in I$ and $v_i=0$ for $i\notin I$. Then $I=\supp(v)$ by construction, so equation \eqref{eq:(vM)_j} tells us that $(vM)_j$ is strictly positive whenever $v_j>0$.

\emph{(ii) $\implies$ (iii):} Suppose that $I$ is nonempty and $M^I$ has an increasing direction. Then by Proposition \ref{prop:dec/inc direction}(ii), all Frobenius principal submatrices of $M^I$ are Hurwitz unstable. It follows that $M^I$ is Hurwitz unstable (see, e.g., Remark \ref{rmk:passing to submatrices}).

\emph{(iii) $\implies$ (ii):} Suppose that $J$ is nonempty and $M^J$ is Hurwitz unstable; that is, $M^J$ has an eigenvalue with strictly positive real part. Consider the Frobenius normal form of $M^J$. As discussed in Remark \ref{rmk:passing to submatrices}, the eigenvalues of $M^J$ are exactly the eigenvalues of the Frobenius principal submatrices of $M^J$, so there must exist a nonempty $I\subseteq J$ such that the Frobenius principal submatrix $(M^J)^I=M^I$ has an eigenvalue with strictly positive real part. But then $M^I$ is an irreducible Hurwitz-unstable matrix, so by Proposition \ref{prop:dec/inc direction}(ii) we conclude that $M^I$ has an increasing direction.

\emph{(iv) $\implies$ (iii):} Obvious ($M$ is a principal submatrix of itself).

\emph{(iii) $\implies$ (iv):} Following the notation of Lemma \ref{lem:Metzler largest eigenvalue}, let $r$ denote the largest eigenvalue of $M$ and let $r'$ denote the maximal largest eigenvalue of a proper submatrix of $M$. Suppose there exists a Hurwitz unstable submatrix of $M$. Either this submatrix is all of $M$ and we are done, or the submatrix is included in the maximum defining $r'$ and so $0<r'$. But by Lemma \ref{lem:Metzler largest eigenvalue} we know $r'\le r$, so $0<r$ and we are done.
\end{proof}

\bibliography{Stochastic_Environments}{}
\bibliographystyle{plain}

\end{document}